\documentclass{amsart}
\usepackage{amsmath, amssymb, amsthm, amsfonts}
\usepackage{graphicx}
\usepackage{tabularx}
\usepackage{hyperref}
\usepackage[utf8]{inputenc}
\usepackage{cleveref}
\usepackage{array}
\usepackage{tikz-cd}

\numberwithin{equation}{section}

\usepackage{enumerate}
\usepackage[font=small]{caption}
\usepackage{subcaption}
\usepackage{cancel}
\usepackage{todonotes}

\newcommand\codim{\textup{codim}\,}

\newcommand\im{\textup{im}\,}

\newcommand\Tor{\textup{Tor}}

\newcommand\pd{\textup{pd}\,}

\newcommand\coker{\textup{coker}\,}

\DeclareMathOperator{\PP}{\mathbb{P}}
\DeclareMathOperator{\reg}{reg}
\DeclareMathOperator{\link}{link}
\DeclareMathOperator{\st}{st}

\DeclareMathOperator{\rel}{\operatorname{rel}}

\theoremstyle{plain}
\newtheorem{Thm}{Theorem}[section]
\newtheorem{Prop}[Thm]{Proposition}
\newtheorem{Cor}[Thm]{Corollary}
\newtheorem{Lem}[Thm]{Lemma}

\theoremstyle{definition}
\newtheorem{Def}[Thm]{Definition}
\newtheorem{Q}[Thm]{Question}
\newtheorem{Notation}[Thm]{Notation}
\newtheorem{Prob}[Thm]{Problem}
\newtheorem{Rmk}[Thm]{Remark}
\newtheorem{Ex}[Thm]{Example}

\newcolumntype{C}[1]{>{\centering\arraybackslash}m{#1}}

\theoremstyle{remark}

\begin{document}

\title{A Fr\"{o}berg type theorem for higher secant complexes}
\author{Junho Choe and Jaewoo Jung${}^\ast$}
\address{Junho Choe \\ School of Mathematics, Korea Institute for Advanced Study
(KIAS), 85 Hoegiro, Dongdaemun-gu, Seoul, 02455, Republic of
Korea}
\email{junhochoe@kias.re.kr}

\address{Jaewoo Jung \\ Global Basic Research Laboratory - Algebra and Geometry of Spaces of Tensors, and Applications (GBRL-AGSTA), Daegu Gyeongbuk Institute of Science and Technology (DGIST), 333 Techno Jungang-daero, Hyeonpung-eup, Dalseong-gun, Daegu 42988, Republic of Korea}
\email{jaewoojung@dgist.ac.kr}

\thanks{${}^\ast$ Corresponding author}

\date{\today}

\begin{abstract}
We generalize the celebrated Fr{\"o}berg's theorem to embedded joins of copies of a simplicial complex, namely higher secant complexes to the simplicial complex, in terms of property $N_{q+1,p}$ due to Green and Lazarsfeld. 
Furthermore, we investigate combinatorial phenomena parallel to geometric ones observed for higher secant varieties of minimal degree.
\end{abstract}

\keywords{Stanley-Reisner ideals, embedded joins of simplicial complexes}
\subjclass[2020]{Primary:~13F55; Secondary:~14N07, 05E45}
\maketitle
\tableofcontents 
\setcounter{page}{1}

\section{Introduction}
A \emph{(abstract) simplicial complex} $\Delta$ on a finite set $V$ is a set of subsets, called \emph{faces}, of $V$ with the transitive property that if $\Delta\ni F\supseteq G$, then $\Delta\ni G$. We impose a nondegeneracy condition that $\{v\}\in\Delta$ for every $v\in V$. Here every element of $V$ is called a \emph{vertex} of $\Delta$, and we write $V=V(\Delta)$ and call it the \emph{vertex set} of $\Delta$. A face of $\Delta$ is called an \emph{edge} if it consists of exactly two vertices. We define a \emph{graph} $G$ on $V$ as a simplicial complex on $V$ whose faces have at most two vertices. We denote a graph $G$ by $G=(V(G),E(G))$, where $V(G)$ is the vertex set of $G$, and $E(G)$ is the \emph{edge set} of $G$, that is, the set of edges of $G$. For a graph $G$ its \emph{clique complex} $\Delta(G)$ is the simplicial complex on $V(G)$ whose faces are precisely cliques of $G$.

The main object of this paper is \emph{embedded joins} of copies of a simplicial complex.

\begin{Def}
    Let $\Delta$ be a finite simplicial complex on the vertex set $V = V(\Delta)$.
    For an integer $q\geq 1$ we define the \emph{$q$-secant complex} $\sigma_q \Delta$ to $\Delta$ as the simplicial complex on $V$ given by
$$
\sigma_q \Delta = \left\{ F_1 \cup \cdots \cup F_q \subseteq V : F_1, \ldots, F_q \text{ are faces of } \Delta \right\}.
$$
\end{Def}

To state our main result we introduce some families of graphs.
\begin{enumerate}
    \item $\mathcal F_{q,1}$ is the collection of graphs $H\neq\overline{K_{q+1}}$ such that $V(H)$ is not a union of $q$ cliques, but so is every proper subset.
    \item $\mathcal F_{q,2}$ consists of $H\sqcup\overline{K_{q-j}}$ for all integers $1\leq j\leq q$ and all \emph{elementary bipartite graphs} $H$ on $2j+2$ vertices.
\end{enumerate}
Here $\overline{K_s}$ means the complement of the complete graph $K_s$ on $s$ vertices, and for the definition and structure of \emph{elementary bipartite graphs}, see \Cref{EB} and \Cref{EBstructure}.
\begin{Def}
Let $G$ be a finite (simple) graph.
It is called \emph{$q$-secant chordal} for an integer $q\geq 1$ if it is $\mathcal F_{q,1}$-free, $\mathcal F_{q,2}$-free, and $C_{2q+p}$-free for all $p\geq 3$.
That is, $G$ does not contain any graph in $\mathcal F_{q,1} \cup \mathcal F_{q,2} \cup \{C_{2q+p} : p\geq 3\}$ as an induced subgraph.
\end{Def}

Let $\Delta$ be a simplicial complex with vertex set $V=\{0,1,\ldots,r\}$ and $S=\Bbbk[x_0,x_1,\ldots,x_r]$ be a polynomial ring in $r+1$ variables over a field $\Bbbk$ with the standard grading.
The \emph{Stanley-Reisner ideal} of $\Delta$ is defined to be
$$ I(\Delta)=(x_W:W\subseteq V\textup{ is not a face of }\Delta)\subseteq S,$$
where $x_W=\prod_{i\in W}x_i$, and its quotient ring $S(\Delta)=S/I(\Delta)$ is called the \emph{Stanley-Reisner ring} of $\Delta$. 
There are active studies on algebraic properties of Stanley-Reisner ideals and rings from various viewpoints.
For instance, see \cite{MR3070118}, \cite{MR2790928}, \cite{MR2943752}, and \cite{MR3249840}. 

Pick a finitely generated graded $S$-module $M$, and let us say that $M$ has minimal (graded) free resolution of the form in \Cref{MFR} so that we have $\coker d_1\cong M$ and $\im d_i\subseteq (x_0,\ldots,x_r)\cdot F_{i-1}$ for all $i\geq 1$. Then the exponents $\beta_{i,j}=\beta_{i,j}(M)$ of the direct summands are homological invariants of $M$ that are called \emph{graded Betti numbers} of $M$. 

\begin{figure}[h!]
$$
\begin{tikzcd}[column sep=large]
F_0 & \ar[l,"d_1",swap] F_1    & \ar[l,"d_2",swap] \cdots & \ar[l,"d_i",swap] F_i \ar[d,equal] & \ar[l] \cdots \\
    &               &               & {\displaystyle \bigoplus_{j\in\mathbb Z}S^{\beta_{i,j}}(-i-j)}
\end{tikzcd}
$$
    \caption{Minimal free resolution}
    \label{MFR}
\end{figure}

Now our main theorem characterizes, by means of \emph{forbidden induced subgraphs}, simplicial complexes $\Delta$ for a fixed integer $p\geq 1$ such that $S(\sigma_q\Delta)$ satisfies \emph{property $N_{q+1,p}$}, that is,
$$
\beta_{i,j}(S(\sigma_q\Delta))=0\quad\text{whenever}\quad i\leq p\textup{ and }j\geq q+1.
$$
Property $N_{q+1,p}$ is a refinement of having \emph{$(q+1)$-linear} resolution in the sense that $S(\sigma_q\Delta)$ satisfies property $N_{q+1,p}$ for all $p\geq 1$ if and only if $I(\sigma_q\Delta)$ has $(q+1)$-linear resolution, that is, $\beta_{i,j}(I(\sigma_q\Delta))=0$ as long as $j\neq q+1$.

\begin{Thm}\label{thm:main}
Let $\Delta$ be a finite simplicial complex with vertex set $V=V(\Delta)$ and $q\geq 1$ be an integer. Then 
\begin{enumerate}
\item\label{Main1} $S(\sigma_q\Delta)$ satisfies property $N_{q+1,1}$ if and only if we have  $\sigma_q\Delta=\sigma_q\Delta(G)$ for some $\mathcal F_{q,1}$-free graph $G$ on $V$,

\item\label{Main2} it also satisfies property $N_{q+1,2}$ if and only if such a graph $G$ is $\mathcal F_{q,2}$-free, and

\item\label{Main3} it also satisfies property $N_{q+1,p}$ for an integer $p\geq 3$ if and only if such a graph $G$ is $C_{2q+i}$-free for all $3\leq i\leq p$.
\end{enumerate}
In particular, $I(\sigma_q\Delta(G))$ has $(q+1)$-linear resolution if and only if $G$ is $q$-secant chordal.
\end{Thm}

The case $q=1$ has already been shown in \cite[Theorem 2.1]{MR2188445}, and the statement \eqref{Main1} follows from a study by Sturmfels and Sullivant \cite{MR2252121}.

One of the motivations is the celebrated \emph{Fr{\"o}berg's theorem} \cite{MR1171260}. It completely characterizes simplicial complexes $\Delta$ such that $I(\Delta)$ has $2$-linear resolution, asserting that for a graph $G$ the Stanley-Reisner ideal $I(\Delta(G))$ has $2$-linear resolution if and only if $G$ is chordal, that is, it has no induced cycles of four or more vertices. The authors of \cite{MR2188445} refined Fr{\"o}berg's result in view of property $N_{2,p}$ in \cite[Theorem 2.1]{MR2188445}: For a graph $G$ the Stanley-Reisner ring $S(\Delta(G))$ satisfies property $N_{2,p}$ for an integer $p\ge 1$ if and only if there are no chordless cycles of at most $p+2$ vertices in $G$.

Another motivation comes from studies about intriguing classes of projective varieties, namely the families of \emph{varieties of minimal higher secant degree}. 
Assume that the base field $\Bbbk$ is algebraically closed of characteristic $0$ for a while. 
Let $X\subseteq\mathbb P^r$ be a projective variety over $\Bbbk$, where a \emph{variety} is always irreducible and reduced, and set it to be \emph{nondegenerate}, that is, $X$ lies in no hyperplanes.
Then there is an elementary lower bound on the degree of $X$ by its codimension: $\deg X \ge \codim X + 1$, and the varieties having $$
\deg X = \codim X + 1
$$ 
are called \emph{varieties of minimal degree}. Regarding varieties of minimal degree, while their geometric classification has been found by del Pezzo and Bertini (cf. \cite{MR0927946}), their algebraic characterizations have been extensively studied. Especially, by Eisenbud and Goto \cite{MR0741934} a variety $X$ is of minimal degree if and only if $I(X)$ has $2$-linear resolution.

We expand our discussion to the \emph{$q$-secant variety} to $X\subseteq\mathbb P^r$ for an integer $q\geq 1$. 
It is defined to be the Zariski closure
$$
\sigma_q X=\overline{\bigcup_{z_i\in X}\langle z_1,\ldots,z_q\rangle}\subseteq\mathbb P^r,
$$
where $\langle z_1,\ldots,z_q\rangle\subseteq\mathbb P^r$ is the subspace spanned by points $z_1,\ldots,z_q$ of $X$. 
In joint papers \cite{MR4441153} and \cite{choe2022determinantal} with Kwak the first author generalized both the del Pezzo-Bertini classification and the Eisenbud-Goto result in consideration of higher secant varieties. 
In detail, $X\subseteq\mathbb P^r$ (resp.\ $\sigma_q X \subseteq \mathbb{P}^r$) is called a \emph{variety of minimal $q$-secant degree} (resp.\ a \emph{$q$-secant variety of minimal degree}) if for the codimension $c = \codim \sigma_q X$ the \emph{Ciliberto-Russo bound}
$$ \deg \sigma_q X \geq \binom{c+q}{q} $$
(\cite[Theorem 4.2]{MR2199628}) turns into equality. Then higher secant varieties of minimal degree have a classification in terms of determinantal presentation, and it holds that $\sigma_qX\subseteq\mathbb P^r$ is a $q$-secant variety of minimal degree if and only if $I(\sigma_qX)$ has $(q+1)$-linear resolution. However, as one may observe in \cite[Theorem 9.1 and Theorem 9.2]{MR2199628} it seems not easy to completely classify the varieties of minimal higher secant degree and would give rise to rich and interesting geometry.

Now we drop the irreducibility condition on $X\subset\mathbb P^r$. 
In this case the notion of variety of minimal degree can be naturally extended by using the Eisenbud-Goto characterization. 
Indeed, as stated in \cite[Theorem 0.4]{MR2275024} the possibly reducible algebraic set $X\subset\mathbb P^r$ has a $2$-linear resolution if and only if its irreducible components are varieties of minimal degree in the span of each and are ``linearly joined" according to a specific ordering of themselves.

In order to explore higher secant varieties in a broader setting, we focus on \emph{coordinate subspace arrangements}, that is, unions of coordinate linear spaces in $\mathbb P^r$. For a coordinate subspace arrangement $Z\subseteq\mathbb P^r$ the homogeneous ideal $I(Z)$ is generated by squarefree monomials, and one can study algebraic properties of $I(Z)$ through the associated simplicial complex $\Delta$, using the \emph{Stanley-Reisner correspondence}. We define the $q$-secant variety $\sigma_qZ\subseteq\mathbb P^r$ to $Z\subseteq\mathbb P^r$ with the same formula as above. 

We remark that the Stanley-Reisner correspondence identifies $\sigma_qZ$ and $\sigma_q\Delta$  (\cite[Corollary 3.3]{MR1749874}). Hence, \Cref{thm:main} resolves a combinatorial counterpart to the classification problem for varieties of minimal higher secant degree, and its last assertion can be said to provide a combinatorial explanation for why varieties of minimal $q$-secant degree seem various and more diverse as $q$ increases, as in \cite[Theorem 9.1 and Theorem 9.2]{MR2199628}.

Below are corollaries of \Cref{thm:main}, regarding property $N_{q+1,p}$ for integers $q\geq 1$ and $p \geq 1$. 

\begin{Cor}\label{MainCor}
Let $\Delta$ be a finite simplicial complex and $q\geq 1$ be an integer. Assume that $S(\sigma_j\Delta)$ satisfies property $N_{j+1,1}$ for all $2\leq j<q$. Then if $S(\sigma_q\Delta)$ satisfies property $N_{q+1,p}$ for an integer $p\geq 3$, then $S(\sigma_{q+1}\Delta)$ satisfies property $N_{q+2,p-2}$. In particular, given an integer $p\geq 1$ the Stanley-Reisner ring $S(\sigma_q\Delta)$ satisfies property $N_{q+1,p-2q+2}$ for every $q\geq 1$ as long as it holds for the case $q=1$.
\end{Cor}

Hence, if $I(\Delta)$ has $2$-linear resolution, then $I(\sigma_q\Delta)$ has $(q+1)$-linear resolution for all $q\geq 2$, which is parallel to the fact (cf.\ \cite{MR2199628}) that if $X\subseteq\mathbb P^r$ is a variety of minimal degree, then $\sigma_qX\subseteq\mathbb P^r$ is a $q$-secant variety of minimal degree for every $q\geq 2$. Also notice that the same vanishing pattern as in the last assertion is observed for sufficiently large degree complete embeddings of smooth irreducible curves due to Ein, Niu, and Park \cite{MR4160876}.

\begin{Cor}\label{InnerProj}
Let $G$ be a finite simple graph and $q\geq 1$ be an integer. 
Then if $G$ is $q$-secant chordal, then so is its edge contraction $G/e$.
More generally, for a finite simplicial complex $\Delta$ if $S(\sigma_q\Delta)$ satisfies property $N_{q+1,p}$ for an integer $p\geq 1$, then $S(\sigma_q(\Delta/e))$ satisfies property $N_{q+1,p-1}$ for every edge $e$ of $\Delta$.
\end{Cor}

A geometric twin of the first assertion is the fact that if $X\subseteq\mathbb P^r$ is a variety of minimal $q$-secant degree for an integer $q\geq 1$, then so is its linear projection from a general point of $X$, namely a general inner projection of $X$.

On the other hand, results in \cite{MR4441153} also suggest a concept of subextremal higher secant varieties beyond varieties of minimal degree. They are called \emph{del Pezzo higher secant varieties}, defined by both degree and sectional genus, and admit a syzygetic characterization: A $q$-secant variety $\sigma_qX\subsetneq\mathbb P^r$ is \emph{del Pezzo} if and only if \Cref{DelPezzo} determines all the values of $\beta_{i,j}=\beta_{i,j}(S(\sigma_qX))$, where 
$$
\beta_{p,q}=\binom{p+q-1}{q}\binom{c+q}{p+q}-\binom{c+q-p-1}{q-1}\binom{c+q-1}{c+q-p}
$$
for integers $0 < p < c= \codim \sigma_q X$, and the remaining entries are all zero except for $\beta_{0,0}=\beta_{c,2q}=1$.
Note that $S(\sigma_q X)$ turns out to be Gorenstein in this case. 

\begin{table}[h!]
\begin{tabular}{c|ccccc}
 $j \setminus i$  & $0$ & $1$             & $\cdots$ & $c-1$             & $c$ \\ \hline
$0$               & $1$ & .             & $\cdots$ & .                & . \\
$q$               & . & $\beta_{1,q}$   & $\cdots$ & $\beta_{c-1,q}$    & . \\
$2q$              & . & .             & $\cdots$ & .                & $1$
\end{tabular}
    \caption{Del Pezzo $q$-secant varieties and complexes}
    \label{DelPezzo}
\end{table}

\begin{Def}\label{Def:delPezzo}
Let us call a $q$-secant complex $\sigma_q\Delta$ \emph{del Pezzo} if the $\beta_{i,j}(S(\sigma_q\Delta))$ are the same as those of a del Pezzo $q$-secant variety.    
\end{Def}
Then for each $q\geq1$ we obtain a classification of graphs whose clique complexes yield del Pezzo $q$-secant complexes.
Recall that a vertex of a simplicial complex $\Delta$ is called \emph{universal} if it belongs to all maximal faces, namely facets, of $\Delta$. Such a vertex has no effects on the graded Betti numbers.

\begin{Thm}\label{MainDelPezzo}
Let $G$ be a finite simple graph and $q\geq 1$ be an integer.
Suppose that $\sigma_q\Delta(G)$ has no universal vertices. Then $\sigma_q\Delta(G)$ is del Pezzo if and only if
\begin{enumerate}
    \item $\codim \sigma_q\Delta(G) = 1$, and $G\in\mathcal F_{q,1}$ with exactly $2q+1$ vertices,
    \item $\codim \sigma_q\Delta(G) = 2$, and $G$ is elementary bipartite on $2q+2$ vertices, or
    \item $\codim \sigma_q\Delta(G) \geq 3$, and $G$ is isomorphic to $C_{2q+c}$, where $c=\codim \sigma_q\Delta(G)$.
\end{enumerate}
\end{Thm}

We emphasize analogues between projective algebraic geometry and combinatorics from the viewpoint of this study. 
In \Cref{Analogue} we list the correspondences between the two areas observed during our investigation. 
It is particularly noteworthy that as highlighted in several papers including \cite{MR2199628}, \cite{MR4441153}, and \cite{choe2022determinantal} taking a general inner projection of $X \subseteq \PP^r$ is a fundamental tool in the study of higher secant varieties $\sigma_q X$. Then contracting an edge of a simplicial complex serves as a counterpart to the general inner projection for our purposes.

\begin{table}[h!]
\centering
\begin{tabular}{c|c}
projective algebraic geometry & combinatorics \\ \hline \hline
projective varieties $X\subseteq\mathbb P^r$ & simplicial complexes $\Delta$ on $\{0,\ldots,r\}$ \\ \hline
homogeneous ideals $I(X)$ & Stanley-Reisner ideals $I(\Delta)$ \\ \hline
varieties of minimal degree & chordal graphs \\ \hline
\cite{MR0741934} & Fr{\"o}berg's theorem \\ \hline \hline
$q$-secant varieties & $q$-secant complexes \\ \hline
varieties of minimal $q$-secant degree & $q$-secant chordal graphs \\ \hline
(cf.\ \cite[Theorems 9.1 and 9.2]{MR2199628}) & \Cref{thm:main} \\ \hline \hline
general inner projections & edge contractions
\end{tabular}
\caption{Analogues between the two areas}
\label{Analogue}
\end{table}

Building on these analogues we explore the combinatorial perspective by investigating the Betti numbers $\beta_{p,q}(\sigma_q\Delta(G))$ and projective dimension of higher secant complexes.
The following theorem establishes a precise combinatorial formula of the Betti numbers.

\begin{Thm}\label{thm:numerical}
Let $G$ be a graph, and suppose that $S(\sigma_q\Delta(G))$ satisfies property $N_{q+1,p-1}$. 
Then, $$\beta_{p,q}(\sigma_q\Delta(G)) = \sum_{s=q}^{p+q-1} \binom{s}{q} |\mathcal{H}_{p+q}^s(G)|,$$
where $\mathcal{H}_n^s(G)$ is the set of induced subgraphs of $G$ with $n$ vertices and $s+1$ connected components.
\end{Thm}

This formula allows us to combinatorially determine the projective dimension of $S(\sigma_q\Delta(G))$ for $q$-secant chordal graphs (\Cref{cor:pd}). 
Moreover, we precisely characterize the condition involving edges under which the projective dimension is preserved under edge contractions (\Cref{cor:prdim}).

Finally, we provide a classification result for $q$-secant chordal forests $G$ such that $S(\sigma_q\Delta(G))$ are \emph{Cohen–Macaulay}.
In other words, we specify forests $G$ satisfying
$$
\pd S(\sigma_q\Delta(G)) = \codim \sigma_q\Delta(G)
$$
under a mild assumption.
\begin{Thm}\label{thm:CMtrees}
Let $G$ be a forest with $|V(G)| \geq 2q+2$, and assume that it has a matching of size $q$ for an integer $q\geq 2$. 
Then $S(\sigma_q\Delta(G))$ is Cohen-Macaulay if and only if $G$ is a path graph.
\end{Thm}

In this context, the path graph on $r+1\geq 2q+2$ vertices is a genuine counterpart of the rational normal curve $C\subset\mathbb P^r$ of degree $r$: the $q$-secant variety $\sigma_qC\subset\mathbb P^r$ has minimal degree, hence $S(\sigma_qC)$ is automatically Cohen-Macaulay \cite[Theorem 1.1]{MR4441153}, and there are no other such projective irreducible curves \cite[Theorem 6.1]{MR2199628}.

This article is organized as follows. 
In \Cref{Sec:Prelim} we introduce necessary notions and review known facts.
In \Cref{Sec:Reg2} we explore higher secant complexes to clique complexes of chordal graphs.
For the families $\mathcal{F}_{q,1}$ and $\mathcal{F}_{q,2}$ playing a central role, their properties are described in \Cref{Sec:special}.
In \Cref{Sec:Main} we state and prove our main results concerning property $N_{q+1,p}$ of $S(\sigma_q \Delta)$ for clique complexes $\Delta$.
In \Cref{Sec:CM} we discuss numerical information of higher secant chordal graphs
and the Cohen-Macaulayness of their higher secant complexes.
In \Cref{Sec:Problems} we present some open problems arising from our study.
\bigskip

\textbf{Acknowledgments}.
J.\ Choe is supported by a KIAS Individual Grant (MG083302) at Korea Institute for Advanced Study. 
J.\ Jung was supported by the Institute for Basic Science (IBS-R032-D1-a00), and has been supported since March 2025 by the National Research Foundation of Korea (NRF) grant funded by the Korean government(MSIT) (RS-2024-00414849).

\bigskip 

\section{Preliminaries}\label{Sec:Prelim}
 We begin this section with introducing some notations that are fixed throughout this article.
\begin{Notation}
\phantom{}
    \begin{enumerate}
    \item $S=\Bbbk[x_0,x_1,\ldots,x_r]$ is a polynomial ring over a field $\Bbbk$ with the standard grading, that is, $\deg x_i = 1$ for any $0\leq i \leq r$.
    \item Given $S=\Bbbk[x_0,x_1,\ldots,x_r]$ we associate to a subset $W\subseteq\{0,1,\ldots,r\}$ a squarefree monomial $x_W=\prod_{i\in W}x_i$.
    \item $\beta_{i,j}(\Delta) := \beta_{i,j}(S(\Delta))$, where $S(\Delta)$ is the Stanley-Reisner ring of the complex $\Delta$.
    \item $K_n$ is the complete graph on $n$ vertices.
    \item $C_n$ is the cycle graph on $n$ vertices.
    \item $P_n$ is the path graph on $n$ vertices.
    \item $\overline{G}$ is the complement of a graph $G$.
    \end{enumerate}
\end{Notation}

\subsection{Graded Betti numbers}
In this subsection we review numerical information and vanishing properties obtained from minimal free resolutions.

\begin{Def}
Given a finitely generated graded $S$-module $M$, \emph{graded Betti numbers} of $M$ are defined to be
$$
\beta_{i,j}=\beta_{i,j}(M)=\dim_\Bbbk\Tor_i^S(M,\Bbbk)_{i+j}
$$
for integers $i,j$, and the \emph{Betti table} of $M$ is defined to be the following diagram.
$$
\begin{array}{c|cccc}
     & \cdots & i & i+1 & \cdots \\ \hline
    \vdots & & \vdots & \vdots & \\
    j & \cdots & \beta_{i,j} & \beta_{i+1,j} & \cdots \\
    j+1 & \cdots & \beta_{i,j+1} & \beta_{i+1,j+1} & \cdots \\
    \vdots & & \vdots & \vdots
\end{array}
$$
For brevity we omit entries if they are zero.
\end{Def}

By nature of Tor functors the graded Betti numbers $\beta_{i,j}=\beta_{i,j}(M)$ fit in the minimal free resolution of $M$ depicted as in \Cref{MFR}.

\begin{Def}
Let $M$ be a finitely generated graded $S$-module. 
\begin{enumerate}
    \item $M$ satisfies \emph{property $N_{d,p}$} for integers $d$ and $p$ if 
    $$
    \beta_{i,j}(M)=0\quad\textup{when}\quad i\leq p \textup{ and } j\geq d.
    $$
    \item The \emph{Castelnuovo-Mumford regularity} of $M$ is defined to be 
    $$
    \reg M = \max\{j\in\mathbb Z : \beta_{i,j}(M)\neq0\textup{ for some }i\geq0\}.
    $$
    \item The \emph{projective dimension} of $M$ is defined to be
    $$
    \pd M=\max\{i\geq 0:\beta_{i,j}(M)\neq0\textup{ for some } j \in \mathbb{Z}\}.
    $$
\end{enumerate}
\end{Def}

We remark that when $I\subsetneq S$ is a homogeneous ideal without forms of degree $q$ for an integer $q\geq 1$, the quotient ring $S/I$ satisfies property $N_{q+1,p}$ for an integer $p\geq 0$ if and only if the first $p+1$ free modules of its minimal free resolution are
$$
S^{\beta_{i,j}(S/I)}(-i-j)
$$
with $0\leq i\leq p$, and if it is the case for every $p\geq0$, then $I$ has \emph{$(q+1)$-linear} resolution, which means that $\beta_{i,j}(I)=0$ for any $i\geq 0$ and $j\neq q+1$.

In order to simplify an argument to appear we recall a formula of Herzog and K{\"u}hl.

\begin{Thm}[{\cite[Theorem 1]{MR743307}}]\label{thm:pure}
Let $M$ be a finitely generated graded $S$-module, and assume that the $S$-module $M$ has pure resolution, that is, for an integer $p\geq 0$ there exist integers $j_0,\ldots,j_p$ such that $\beta_{0,j_0}(M),\ldots,\beta_{p,j_p}(M)$ are the only nonzero graded Betti numbers of $M$. 
Then we have
$$
\beta_{i,j_i}(M)=\left|\prod_{k\neq 0,i}\frac{j_k-j_0+k}{j_k-j_i+k-j}\right|\cdot\beta_{0,j_0}(M)
$$
for all $0\leq i\leq p$ if and only if $M$ is Cohen-Macaulay, that is, $\pd M=\codim M$.
\end{Thm}

\subsection{Simplicial complexes and the Stanley-Reisner correspondence}\label{subsec:StanleyReisner}
In this subsection we review (finite abstract) simplicial complexes and their algebraic aspects.
A \emph{simplicial complex} on a finite set $V$ is a collection of subsets of $V$ that is closed under taking subsets.
That is, for any $F \in \Delta$, if $G \subseteq F$, then $G \in \Delta$. 
We assume a nondegeneracy setting that all singletons in $V$ are contained in $\Delta$. 
Any element of $V$ is called a \emph{vertex} of $\Delta$, the set $V$ is called the \emph{vertex set} of $\Delta$, and the members of $\Delta$ are called \emph{faces}. 
An \emph{edge} is a face with only two vertices. 
Faces that are maximal with respect to inclusion are called \emph{facets}. 
A \emph{simplex} on $V$ is the simplicial complex $2^V$, the power set of $V$.
For a vertex $v$ of $\Delta$ if no edges contain $v$, then $v$ is called \emph{isolated}.
For a face $F$ of $\Delta$ we define the \emph{dimension} of $F$ by $\dim F =|F|-1$ and the \emph{dimension} of $\Delta$ by the maximal dimension of faces of $\Delta$.
A \emph{graph} $G$ is a simplicial complex of dimension at most one, and we write $G=(V(G),E(G))$, where $E(G)$ is the set of edges in $G$.
For a subset $W\subseteq V$ the \emph{induced subcomplex} of $\Delta$ on $W$, denoted by $\Delta[W]$, is the simplicial complex on $W$ whose faces are precisely those of $\Delta$ that are contained in $W$. 

Now we recall the \emph{Stanley-Reisner correspondence}. Suppose the following.
\begin{enumerate}
    \item $\Delta$ is a simplicial complex with the vertex set $V=\{0,1,\ldots,r\}$.
    \item $I \subseteq S = \Bbbk[x_0, x_1, \ldots, x_r]$ is a squarefree monomial ideal with no linear generators.
    \item $Z\subseteq\mathbb P^r$ is a nondegenerate coordinate subspace arrangement.
\end{enumerate}
Then $I$ is called the \emph{Stanley-Reisner ideal} of $\Delta$ if it is equal to
$$
I(\Delta)=(x_W:W\subseteq V\textup{ is not a face of }\Delta)\subseteq S.
$$ 
Conversely, $\Delta$ is called the \emph{Stanley-Reisner complex} of $I$ if it is given as
$$
\{F\subseteq V:x_F\not\in I\}.
$$
This coupling between $I$ and $\Delta$ is known as the \emph{Stanley-Reisner correspondence}.
In addition, the ideal-variety correspondence enables us to compare $Z$ with $I$ and so with $\Delta$.

$$
\begin{tikzcd}[column sep=large]
        \Delta \ar[r,"I(\Delta)",mapsto,shift left] & I \ar[l,"\Delta(I)",mapsto,shift left] \ar[r,leftrightarrow] & Z
\end{tikzcd}
$$

One important aspect of the Stanley--Reisner correspondence is that topological data of a simplicial complex $\Delta$ are encoded in algebraic data of the Stanley-Reisner ideal $I(\Delta)$, namely Hochster's formula \cite{MR0441987}. We adopt its homological version.

\begin{Thm}\cite[Corollary 5.12 and Corollary 1.40]{MR2110098}\label{thm:Homology}
    Let $\Delta$ be a simplicial complex on $V$. Then we have 
    $$
     \beta_{i,j}(\Delta) = \sum_{\substack{W\subseteq V, \\ |W|=i+j}} \dim_\Bbbk \widetilde{H}_{j-1}(\Delta[W],\Bbbk)
    $$ 
    for all integers $i$ and $j$ with $i+j>0$.
\end{Thm}

A closed subscheme $Z\subseteq\mathbb P^r$ is called \emph{$(q+1)$-regular} for an integer $q\geq 1$ if its homogeneous ideal $I(Z)$ satisfies $\reg I(Z)\leq q+1$. 
This notion is generalized and interpreted for simplicial complexes as follows.

\begin{Def}
Let $\Delta$ be a simplicial complex on $V$. Then one says that $\Delta$ is \emph{$(q+1)$-regular} for an integer $q\geq 1$ if one of the following equivalent conditions holds.
\begin{enumerate}
    \item $\reg I(\Delta)\leq q+1$.
    \item $\widetilde{H}_j(\Delta[W],\Bbbk)=0$ for all nonempty subsets $W\subseteq V$ and all integers $j\geq q$.
\end{enumerate}
\end{Def}

\subsection{Higher secant complexes}
Given a nondegenerate reduced subscheme $Z\subseteq\mathbb P^r$ and an integer $q\geq 1$ the \emph{$q$-secant variety} to $Z\subseteq\mathbb P^r$ is defined to be the Zariski closure
$$ 
\sigma_qZ=\overline{\bigcup_{z_i\in Z}\langle z_1,\ldots,z_q\rangle}\subseteq\mathbb P^r, 
$$
where $\langle z_1,\ldots,z_q\rangle\subseteq\mathbb P^r$ is the subspace of $\mathbb P^r$ spanned by $z_1,\ldots,z_q \in Z$.

If $Z\subseteq\mathbb P^r$ is a nondegenerate coordinate subspace arrangement, then so is its $q$-secant variety $\sigma_q Z$. Then the Stanley-Reisner complex assigned to $I(\sigma_qZ)$ can be obtained via \emph{embedded join}.

\begin{Def}
For finite subsets $V_1$ and $V_2$ in a set $V$ let $\Delta_1$ and $\Delta_2$ be simplicial complexes on $V_1$ and on $V_2$, respectively. Then the \emph{embedded join} of $\Delta_1$ and $\Delta_2$ is a simplicial complex on $V_1\cup V_2\subseteq V$ defined to be
$$ 
\Delta_1\ast\Delta_2=\{F_1\cup F_2\subseteq V_1\cup V_2:F_1\in\Delta_1\text{ and }F_2\in\Delta_2\}. 
$$
For a simplicial complex $\Delta$ the \emph{$q$-secant complex} to $\Delta$ is defined to be
$$
\sigma_q\Delta=\underbrace{\Delta\ast\cdots\ast\Delta}_{q \text{ times}}=\{F_1\cup\cdots\cup F_q\subseteq V(\Delta):F_i\in\Delta\}. 
$$
\end{Def}

The statement just above is formally written as follows.

\begin{Prop}\cite[Corollary 3.3.]{MR1749874}
    Let $\Delta_1$ and $\Delta_2$ be simplicial complexes on the same vertex set. Then $I(\Delta_1 \ast \Delta_2)$ is the join algebra \cite[Definition 1.1]{MR1749874} of $I(\Delta_1)$ and $I(\Delta_2)$. In particular, for a simplicial complex $\Delta$ on $r+1$ vertices and its corresponding coordinate subspace arrangement $Z\subseteq\mathbb P^r$ we have 
    $$
    I(\sigma_q \Delta)=I(\sigma_q Z)
    $$
    for all $q\geq 1$.
\end{Prop}

As its consequence $\sigma_qZ$ shares known results with $\sigma_q\Delta$. Useful examples are the following.

\begin{Thm}[cf.\ {\cite[Theorem 1.2]{MR2541390}}]\label{thm:Prolong}
Let $\Delta$ be a simplicial complex with the vertex set $V=\{0,1,\ldots,r\}$ and $q\geq 1$ be an integer.
Then the graded pieces of $I(\sigma_q\Delta)$ in degrees $q$ and $q+1$ are $I(\sigma_q\Delta)_q=0$ and
$$
I(\sigma_q\Delta)_{q+1}=\langle x_W:W\subseteq V,\textup{ }|W|=q+1,\textup{ and }\Delta[W]\textup{ has no edges}\rangle.
$$
\end{Thm}

\begin{proof}
    Let $W\subseteq V$ be a subset. Observe that it must be a face of $\sigma_q\Delta$ when $|W|=q$ and that when $|W|=q+1$, it is a face of $\sigma_q\Delta$ if and only if $\Delta$ has at least one edge of two vertices in $W$.
\end{proof}

Taking higher secant complexes is not a one-to-one process. The following resolves some subtleties.

\begin{Lem}[cf.\ {\cite[Lemma 1.4]{MR2392585}}]\label{Eventually}
Let $\Delta_1$ and $\Delta_2$ be simplicial complexes with the same vertex set such that $\Delta_1\subseteq\Delta_2$, and take an integer $q\geq 1$. Then $\sigma_q\Delta_1=\sigma_q\Delta_2$ implies that $\sigma_{q+1}\Delta_1=\sigma_{q+1}\Delta_2$.  
\end{Lem}

\begin{proof}
    We trivially have $\sigma_{q+1}\Delta_1\subseteq\sigma_{q+1}\Delta_2$. For the reverse containment let $F=F_1\cup\cdots\cup F_{q+1}\in\sigma_{q+1}\Delta_2$ be an arbitrary face, where $F_1,\ldots,F_{q+1}\in\Delta_2$. Then since $F_1\cup\cdots\cup F_q\in\sigma_q\Delta_2=\sigma_q\Delta_1$, there are faces $G_1,\ldots,G_q\in\Delta_1$ such that $G_1\cup\cdots\cup G_q=F_1\cup\cdots\cup F_q$, hence $F=G_1\cup\cdots\cup G_q\cup F_{q+1}$. Do the same thing for $G_2\cup\cdots\cup G_q\cup F_{q+1}$ so that it is equal to $H_2\cup\cdots\cup H_{q+1}$ for some $H_2,\ldots,H_{q+1}\in\Delta_1$. We conclude that $F=G_1\cup H_2\cup\cdots\cup H_{q+1}\in\sigma_{q+1}\Delta_1$.
\end{proof}

We collect operations on simplicial complexes. Let $\Delta$ be a simplicial complex on $V$.
\begin{enumerate}
    \item The \emph{induced subcomplex} of $\Delta$ by a subset $W\subseteq V$ is defined to be  
    $$\Delta[W] = \{ F \cap W : F \in \Delta \}.$$
    \item For a face $F$ in $\Delta$, the \emph{star} and \emph{link} of $F$ in $\Delta$ are defined by
    $$\st_\Delta (F)= \{G\subseteq V: G\cup F\in\Delta\}
    $$ 
    and 
    $$\link_\Delta (F) = \{ G\subseteq V\setminus F: G\cup F \in \Delta\},
    $$
    respectively.
    Note that $\st_\Delta (F) = \{ F'\cup F: F' \in \link_\Delta (F)\}$.
    \item The \emph{vertex deletion} of $\Delta$ by a vertex $v$ of $\Delta$ is defined by $$\Delta- v = \{F \subseteq V\setminus \{v\}: F \in \Delta\}.$$
\end{enumerate}

We remark that taking induced subcomplexes commutes with taking higher secant complexes. 
That is, 
$$
(\sigma_q \Delta)[W] = \sigma_q(\Delta[W]).
$$
In particular, vertex deletions commute with taking higher secant complexes.

We recall face contractions of complexes.
\begin{Def}
    Let $\Delta$ be a simplicial complex on $V$ and $F$ be a face of $\Delta$, and introduce a new vertex $v$. On $(V\setminus F)\cup\{v\}$ we define a simplicial complex
    $$
    \Delta/F=\{(G\setminus F)\cup\{v\}: G\cap F\neq\emptyset\}\cup\{G:G\cap F=\emptyset\},
    $$
    where $G$ stands for a face of $\Delta$, and we call it the \emph{face contraction} of $\Delta$ by $F$. If $F$ is an edge, then it is called an \emph{edge contraction}.
\end{Def}

\begin{Rmk}
Let $\Delta$ be a simplicial complex.
    \begin{enumerate}
        \item For any face $F\in\Delta$ one can see that $\sigma_q(\Delta/F)=(\sigma_q\Delta)/F$, hence $\sigma_q\Delta/F$ is well-defined.
        \item When $\Delta$ is identified with a coordinate subspace arrangement $Z$ under the Stanley-Reisner correspondence, one can understand an edge contraction $\Delta/e$ as the linear projection of $Z$ from a general point in the projective line corresponding to the edge $e$.
    \end{enumerate}
\end{Rmk}

Edge contractions present a fascinating observation that plays a central role in our arguments to appear.

\begin{Thm}[{\cite[Theorem 2 and Lemma 3]{MR2919643}}]\label{lkcond}
Let $e$ be an edge of a simplicial complex $\Delta$, and suppose that $e$ is not contained in any minimal nonface of $\Delta$. Then $\Delta/e$ is homotopy equivalent to $\Delta$.
\end{Thm}

Combining with \Cref{thm:Prolong} we reach the following.

\begin{Cor}\label{cor:EdgePreserve}
Let $\Delta$ be a simplicial complex and $q\geq 1$ be an integer. Assume that $I(\sigma_q\Delta)$ is generated in degree $q+1$. Then for any edge $e$ of $\Delta$ the natural simplicial map $\phi:\Delta\to\Delta/e$ induces isomorphisms
$$
\phi_\ast:\widetilde{H}_\ast(\sigma_q\Delta,\Bbbk)\to\widetilde{H}_\ast(\sigma_q\Delta/e,\Bbbk).
$$
\end{Cor}

Of particular interest in this article are simplicial complexes that arise from graphs. 
Let $G$ be a (finite simple) graph with vertex set $V$. 
Then we define the \emph{clique complex} $\Delta(G)$ of $G$ by 
$$
\Delta(G) = \{F\subseteq V : F \text{ is a clique of }G \}
$$
Recall that by definition a subset $F\subseteq V$ is a clique of $G$ if $G[F]$ is a complete graph.

\begin{Rmk}
    Let $\Delta(G)$ be a clique complex.
    \begin{enumerate}
        \item $I(\Delta(G))$ is generated in degree $2$ since minimal nonfaces of $\Delta(G)$ consist of two vertices.
        \item Taking induced subgraphs commutes with taking clique complexes. 
        \item But edge contraction does not commute with taking clique complexes in general. The commutativity holds if and only if the edge contracted lies in no induced cycles $C_4$.
    \end{enumerate}
\end{Rmk}

Due to the correspondence between a graph $G$ and its clique complex $\Delta(G)$, operations on simplicial complexes are naturally inherited by graphs.  
Thus, we now introduce some notions from structural graph theory.

Let $G = (V, E)$ be a graph, $v$ be a vertex in $G$, and $W$ be a subset of $V$.
\begin{enumerate}
    \item The \emph{induced subgraph} of $G$ by a subset $W\subseteq V$ is defined by $G[W]=(W,E')$, where $E' = \{e : e\in E(G) \text{ with } e\subseteq W\}$.
    \item The induced subgraph $G-v = G[V \setminus \{v\}]$ is called the \emph{vertex deletion} of $G$ by $v$.
    \item The \emph{open neighborhood} of $v$ in $G$, denoted by $N_G(v)$, is the subgraph of $G$ induced by vertices in $V$ adjacent to $v$.  
    The \emph{degree} of $v$ in $G$ is $\deg_G v = |N_G(v)|$.
    \item The \emph{closed neighborhood} of $v$ in $G$, denoted by $N_G[v]$, is the subgraph of $G$ induced by $v$ itself and the vertices in $V$ adjacent to $v$.  
    (That is, $N_G[v] = G[\{v\} \cup W]$, where $W = \{ w \in V : \{v, w\} \in E(G) \}$.)
    \item A graph $G$ is said to be \emph{$H$-free} for another graph $H$ if no induced subgraphs of $G$ are isomorphic to $H$. 
    \item For a family $\mathcal F$ of graphs one says that $G$ is \emph{$\mathcal F$-free} if it is $H$-free for all $H\in\mathcal F$.
\end{enumerate}

Now we discuss some results from structural graph theory.
The following notions take part in the next sections. 

\begin{Def}\label{def:Matching}
Let $G=(V,E)$ be a graph on $V$. 
\begin{enumerate}
    \item A \emph{matching} of size $q$ is a set of $q$ edges in $G$ sharing no vertices with each other.
    \item The \emph{matching number} of $G$, denoted by $\nu(G)$, is defined as the cardinality of the largest matching in $G$. 
    In other words, $\nu(G)$ is the size of the largest independent edge set in $G$.
    \item A matching $M$ is called \emph{perfect} if every vertex in $V$ is incident to exactly one edge in $M$. 
    \item A subset $U$ of $V$ is called a \emph{vertex cover} if every edge in $G$ is incident to at least one vertex in $U$.
    \item A vertex cover $U$ of $V$ is called a \emph{minimum vertex cover} if in addition the cardinality of $U$ is smallest among all vertex covers of $G$.
\end{enumerate}
\end{Def}

\bigskip

\section{Higher secant complexes to 2-regular complexes} \label{Sec:Reg2}

In this section we extend the following result of Fr{\"o}berg to higher secant complexes to clique complexes of chordal graphs.
\begin{Thm}[{\cite[Theorem 1]{MR1171260}}]\label{thm:Froeberg}
    A simplicial complex $\Delta$ is $2$-regular if and only if $\Delta=\Delta(G)$ for some chordal graph $G$ on $V(\Delta)$.
\end{Thm}

Useful facts about the structure of chordal graphs are that they are closed under vertex deletions and that any chordal graph $G$ admits a vertex $v$ such that $N_G(v)$ form a clique, namely a \emph{simplicial vertex}.
\begin{Thm}[A corollary of {\cite[Theorem 1]{MR130190}}]
Every chordal graph $G$ has a simplicial vertex.
\end{Thm}
Hence, any chordal graph $G$ has a \emph{perfect elimination ordering} $(v_0,v_1,\dots,v_r)$ of the vertices in $G$ so that for every $i=0,1,\ldots,r$ the vertex $v_i$ is simplicial in the vertex deletion $G-v_0-v_1-\cdots-v_{i-1}$ being chordal.

By the structure of chordal graphs one can prove the following direction of Fr{\"o}berg's result by using the Mayer-Vietoris sequence and induction on the number of vertices: If $G$ is chordal, then $\Delta(G)$ is $2$-regular. 
For a generalization to \emph{hereditary} families of graphs, see \cite[Theorem 1.4]{MR4123750}.

The following is the main theorem of this section.

\begin{Thm}\label{thm:2Regular}
If a simplicial complex $\Delta$ is $2$-regular, then $\sigma_q\Delta$ is $(q+1)$-regular for every integer $q\geq 2$. 
\end{Thm}

\begin{proof}
We proceed by induction on the number of vertices. Note that \Cref{thm:Froeberg} gives a chordal graph $G$ on $V=V(\Delta)$ such that $\Delta(G)=\Delta$. Take a simplicial vertex $v$ of $G$. We decompose $\sigma_q\Delta$ into
$$
\sigma_q \Delta = (\sigma_q \Delta - v) \cup (\st_{\sigma_q \Delta} v)\quad\text{with}\quad(\sigma_q \Delta - v) \cap (\st_{\sigma_q \Delta} v) = \link_{\sigma_q \Delta} v.
$$
By the Mayer-Vietoris sequence we obtain 
$$
\reg(\sigma_q \Delta) \le \max \{ \reg(\sigma_q \Delta - v), \reg(\link_{\sigma_q \Delta} v) + 1\}.
$$
So it suffices to verify that $\reg(\sigma_q\Delta(G)-v)\leq q+1$ and $\reg(\link_{\sigma_q\Delta(G)}v)\leq q$. 

To this end observe that $\sigma_q\Delta(G)-v=\sigma_q\Delta(G-v)$ and 
$$
\link_{\sigma_q\Delta(G)}v=(\link_{\Delta(G)}v)\ast(\sigma_{q-1}\Delta(G)-v)=\Delta(G[N_G(v)])*\sigma_{q-1}\Delta(G-v).
$$
Especially, $\link_{\sigma_q\Delta(G)}v$ is isomorphic to $\sigma_{q-1}\Delta(G[V\setminus N_G[v]])$ up to universal vertices, for $v$ is simplicial in $G$. 
Since $G-v$ and $G[V\setminus N_G[v]]$ are both chordal with less vertices, we are done by the induction hypothesis.
\end{proof}

The same phenomenon occurs in the class of projective varieties. 
Explicitly speaking, if $X\subseteq\mathbb P^r$ is a variety of minimal degree over $\Bbbk$ algebraically closed, then $\sigma_qX\subseteq\mathbb P^r$ is a $q$-secant variety of minimal degree for any integer $q\geq 2$. See for example \cite{MR2199628}. 

Due to \Cref{thm:main} we will see that if $G$ is a chordal graph, then it is also $q$-secant chordal for any $q \geq 2$.

The converse of \Cref{thm:2Regular} is not true obviously. We provide a counterexample.

\begin{Ex}\label{ex:sunlet}
    Let $\Delta$ be the clique complex of the graph in \Cref{fig:3regbutminimalsecants}.
    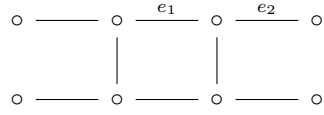
\begin{figure}[h!]
        \centering
        \begin{tikzcd}
        \circ \ar[r,no head] & \circ \ar[r,no head,"e_1"] \ar[d,no head] & \circ \ar[r,no head,"e_2"] \ar[d,no head] & \circ \\
        \circ \ar[r,no head] & \circ \ar[r,no head] & \circ \ar[r,no head] & \circ 
        \end{tikzcd}
        \caption{The $4$-sunlet graph}
        \label{fig:3regbutminimalsecants}
    \end{figure}
    
With assistance of Macaulay2 \cite{M2} we see that 
$$
\reg\Delta=3,\quad\text{but}\quad\reg\sigma_2\Delta=3,\quad\text{and}\quad\reg\sigma_3\Delta=4
$$
with $\sigma_4\Delta$ a simplex.
\end{Ex}

As a result we may explicitly compute the whole graded Betti numbers of higher secant complexes to certain $2$-regular complexes.

\begin{Cor}\label{cor:PKBetti}
Suppose that $G$ is either $P_{r+1}$ or $\overline{K_{r+1}}$, and let $c$ denote the codimension of $\sigma_q\Delta(G)$.
Then $S(\sigma_q\Delta(G))$ is Cohen-Macaulay, and we obtain
$$\beta_{p,q}(\sigma_q\Delta(G))= \binom{p+q-1}{q}\binom{c+q}{p+q}\quad\text{for}\quad 1\le p \le c.$$
\begin{table}[h!]
    \centering
    \begin{tabular}{c|c c c c c c}
          & $0$ & $1$ & $2$ & $\cdots$ & $c$ \\
        \hline
       $0$ & $1$ & . & . & $\cdots$ & . \\
       $q$ & . & $\beta_{1,q}$ & $\beta_{2,q}$ & $\cdots$  & $\beta_{c,q}$ \\
    \end{tabular}
    \caption{Betti tables of $\sigma_q\Delta(\overline{K}_{r+1})$ and $\sigma_q\Delta(P_{r+1})$}
    \label{tab:secantisolatedpts}
\end{table}
\end{Cor}

\begin{proof}
It suffices to show that $S(\sigma_q\Delta(G))$ is Cohen-Macaulay.
Indeed, since both $\overline{K_{r+1}}$ and $P_{r+1}$ are chordal, $S(\sigma_q\Delta(G))$ has a pure resolution by \Cref{thm:2Regular}.
Now, if $S(\sigma_q\Delta(G))$ is Cohen-Macaulay, then the Betti number computation follows from applying \Cref{thm:pure}.

We prove that $S(\sigma_q \Delta(\overline{K_{r+1}}))$ is Cohen-Macaulay.  
We may assume that $r \geq q$ since if $r \leq q-1$, then $\sigma_q\Delta(\overline{K_{r+1}})$ is simply a simplex on $r+1$ vertices.
For $r \geq q$ the complex $\sigma_q\Delta(\overline{K_{r+1}})$ is the collection of subsets of $V(\overline{K_{r+1}})$ whose sizes are at most $q$.
In particular, the codimension of $\sigma_q \Delta(\overline{K_{r+1}})$ is $c=r-q+1$.  
By \Cref{thm:Homology} it suffices to show that $\widetilde{H}_{q-1}(\sigma_q \Delta(\overline{K_{r+1}})[W],\Bbbk) = 0$ for all $W \subseteq V(\overline{K_{r+1}})$ with $|W| = c+q+1$.  
This follows immediately because $|W| = r+2 > r+1$.  
Thus, $S(\sigma_q \Delta(\overline{K_{r+1}}))$ is Cohen-Macaulay.
We now prove by induction on $r$ that $S(\sigma_q\Delta(P_{r+1}))$ is Cohen-Macaulay for any fixed integer $q\geq 1$.  
We may set $r \geq 2q$ as the case $r \leq 2q-1$ is trivial since $\sigma_q\Delta(P_{r+1})$ becomes a simplex on $r+1$ vertices.  
If $r = 2q$, then $\sigma_q\Delta(P_{2q+1})$ is given by a single nonface, and so $S(\sigma_q\Delta(P_{2q+1}))$ is Cohen-Macaulay.

Now, suppose $S(\sigma_q\Delta(P_{r}))$ is Cohen-Macaulay for an integer $r\geq 2q+1$.
To show that $S(\sigma_q\Delta(P_{r+1}))$ is Cohen-Macaulay, we claim that $\widetilde{H}_{q-1}(\sigma_q \Delta(P_{r+1})[W],\Bbbk) = 0$ for any $W \subseteq V(P_{r+1})$ with $|W| = c+q+1 =r-q+2$.
Choose such a subset $W$.  
Since $|W| \geq \lceil(r+1)/2\rceil+1$, there exists an edge $e$ in $P_{r+1}[W]$.  
Then by \Cref{cor:EdgePreserve} one sees that
$$
\widetilde{H}_{q-1}(\sigma_q \Delta(P_{r+1})[W],\Bbbk) \cong \widetilde{H}_{q-1}(\sigma_q\Delta(P_{r+1})[W]/e,\Bbbk),
$$
and $\Delta(P_{r+1})[W]/e$ is isomorphic to an induced subcomplex of $\Delta(P_{r})$ with $(c-1)+q$ vertices.
Thus, $\widetilde{H}_{q-1}(\sigma_q \Delta(P_{r+1})[W],\Bbbk)=0$ by the inductive hypothesis.
\end{proof}

This argument is expanded in \Cref{Sec:CM} so that the quantities $\beta_{p,q}(\sigma_q\Delta(G))$ are interpreted graph-theoretically.

\section{Special families of graphs}\label{Sec:special}
In this section we introduce two families of graphs that obstruct $q$-secant complexes from satisfying property $N_{q+1,p}$ for each $q \geq 1$, and also we collect structural results to be used in the proofs of our main theorems.

\begin{Def}[\cite{hetyei2x} and \cite{MR0498248}]\label{EB}
A bipartite graph $G$ with bipartition $V(G)=U\sqcup W$ is called \emph{elementary} if one of the following equivalent conditions holds.
\begin{enumerate}
    \item\label{EB1} $|U|=|W|$, and every subset $\emptyset\subsetneq U'\subsetneq U$ satisfies $|N_G(U')|>|U'|$.
    \item\label{EB2} $G-u-w$ has a perfect matching for all $u\in U$ and $w\in W$. 
    \item \label{EB3} $G$ is connected and equal to the union of perfect matchings of $G$.
    \item \label{EB4} The union of perfect matchings of $G$ forms a connected subgraph.
    \item \label{EB5} $U$ and $W$ are the only minimum vertex covers of $G$.
\end{enumerate}
\end{Def}

\begin{Rmk}
Let $G\subseteq K_{q+1,q+1}$ be an elementary bipartite graph on $2q+2$ vertices, where $K_{q+1,q+1}$ is the complete bipartite graph. Then by \Cref{EB}\eqref{EB1} every subgraph between them is also elementary bipartite.
\end{Rmk}

For the structure of elementary bipartite graphs, we refer to the following. 

\begin{Thm}[{\cite[Theorem 2]{MR0498248}}]\label{EBstructure}
An elementary bipartite graph is precisely a balanced bipartite graph with a unique bipartition of the vertex set that can be recursively constructed as follows.
\begin{enumerate}
    \item $K_2$ is elementary bipartite. 
    \item For an elementary bipartite $H$ with bipartition coloring, a new elementary bipartite graph is obtained by introducing a new path $P_{2m}$ for an integer $m\geq 1$ and then by identifying the endpoints of $P_{2m}$ with any two vertices of different colors in $H$.
\end{enumerate}
\end{Thm}

For examples of elementary bipartite graphs see the following.
\begin{Ex}
    The $4$-cycle $C_4$ is the unique elementary bipartite graph on $4$ vertices.
    We list all elementary bipartite graphs on $6$ vertices in \Cref{fig:EB6} and all edge-minimal elementary bipartite graphs on $8$ vertices in \Cref{fig:EB8}:
    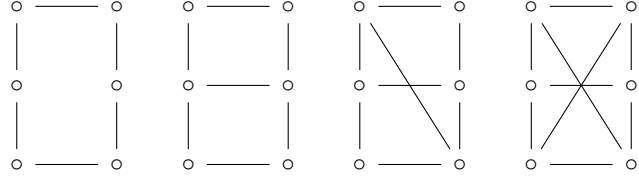
\begin{figure}[h!]
    \centering
    \begin{tikzcd}
        \circ \ar[r,no head] \ar[d,no head] & \circ \ar[d,no head] \\
        \circ \ar[d,no head] & \circ \ar[d,no head] \\
        \circ \ar[r,no head] & \circ
    \end{tikzcd}
    \quad
    \begin{tikzcd}
        \circ \ar[r,no head] \ar[d,no head] & \circ \ar[d,no head] \\
        \circ \ar[d,no head] \ar[r,no head] & \circ \ar[d,no head] \\
        \circ \ar[r,no head] & \circ
    \end{tikzcd}
    \quad
    \begin{tikzcd}
        \circ \ar[r,no head] \ar[d,no head] \ar[ddr,no head] & \circ \ar[d,no head] \\
        \circ \ar[d,no head] \ar[r,no head] & \circ \ar[d,no head] \\
        \circ \ar[r,no head] & \circ
    \end{tikzcd}
    \quad
    \begin{tikzcd}
        \circ \ar[r,no head] \ar[d,no head] \ar[ddr,no head] & \circ \ar[d,no head] \ar[ddl,no head] \\
        \circ \ar[d,no head] \ar[r,no head] & \circ \ar[d,no head] \\
        \circ \ar[r,no head] & \circ
    \end{tikzcd}
    \caption{All elementary bipartite graphs on $6$ vertices}
    \label{fig:EB6}
\end{figure}
\begin{figure}[h!]
\centering
    \begin{subfigure}{0.4\textwidth}
    $$
\begin{tikzcd}
    \circ \ar[r,no head] \ar[d,no head] & \circ \ar[r,no head] & \circ \ar[r,no head] & \circ \ar[d,no head] \\
    \circ \ar[r,no head] & \circ \ar[r,no head] & \circ \ar[r,no head] & \circ
\end{tikzcd}
$$
\caption{}
\end{subfigure}
\begin{subfigure}{0.4\textwidth}
    $$
\begin{tikzcd}
    \circ \ar[d,no head] \ar[r,no head] & \circ \ar[r,no head] & \circ \ar[d,no head] & \circ \ar[d,no head] \ar[lll,no head,bend right=20] \\
    \circ \ar[r,no head] & \circ \ar[r,no head] & \circ \ar[r,no head] & \circ
\end{tikzcd}
$$
\caption{}
\label{NoHamil}
\end{subfigure}
\caption{All edge-minimal elementary bipartite graphs on $8$ vertices}
\label{fig:EB8}
\end{figure}
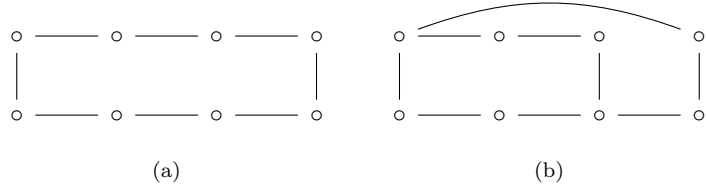
\end{Ex}

\begin{Rmk}
Even cycles are typical examples of elementary bipartite graphs, and many elementary bipartite graphs admit Hamiltonian cycles. 
But it is not the case in general.
For example, consider the graph in \Cref{NoHamil}. 
It is in fact a unique edge-minimal elementary bipartite graph on $8$ vertices other than $C_8$. 
Consult with \cite[Section 2]{MR0498248} for edge-minimal elementary bipartite graphs.
\end{Rmk}

Now we specify graphs having a central role in our study.

\begin{Def}
Let $q\geq 1$ be an integer.
    \begin{enumerate}
    \item $\mathcal F_{q,1}$ is the collection of graphs $H\neq\overline{K_{q+1}}$ such that $V(H)$ is not a union of $q$ cliques, but so is every proper subset (or equivalently the complement graph $\overline{H}$ is not colorable with $q$ colors, but so are its proper induced subgraphs).
    \item $\mathcal F_{q,2}$ consists of $H\sqcup\overline{K_{q-j}}$ for all $1\leq j\leq q$ and all elementary bipartite graphs $H$ on $2j+2$ vertices.
\end{enumerate}
\end{Def}

In the study of higher secant complexes the importance of the families $\mathcal F_{q,1}$ and $\mathcal F_{q,2}$ can be understood via the following observations.

\begin{Thm}[{\cite[Theorem 3.2]{MR2252121}}]\label{Gen}
Let $G$ be a graph on $V = \{0, 1, \ldots, r\}$. 
Then $I(\sigma_q\Delta(G))$ is minimally generated by $x_W$ for all $W\subseteq V$ such that either 
$$
G[W]=\overline{K_{q+1}}\quad\text{or}\quad G[W]\in\mathcal F_{q,1}.
$$
\end{Thm}

\begin{Thm}[{\cite[Proposition 5.1]{MR1749874}}]\label{Gen_Secant}
The family $\mathcal F_{2,1}$ consists of $\overline{C_n}$ with $n\geq 5$ odd.
\end{Thm}

The following also serves as a criterion for determining whether a given graph is elementary bipartite by means of higher secant complexes.

\begin{Prop}\label{EBci}
Let $G$ be a graph on $2q+2$ vertices for an integer $q\geq 1$. Then the following are equivalent.
\begin{enumerate}
    \item $G$ is elementary bipartite.
    \item $I(\sigma_q\Delta(G))$ is a complete intersection ideal of two monomials of degree $q+1$.
    \item $I(\sigma_q\Delta(G))_{q+1}$ is spanned by two monomials without common variables.
\end{enumerate}
Furthermore, in this case $G$ is $\mathcal F_{q,1}$-free, the two monomials correspond to the bipartition classes of $G$, and $I(\sigma_{q+1}\Delta(G))$ is the zero ideal.
\end{Prop}

\begin{proof}
Suppose that $G$ is elementary bipartite with bipartition $V(G)=U\sqcup W$. 
By \Cref{Gen} and \Cref{EB}\eqref{EB1}, it is enough to show that $G$ is $\mathcal F_{q,1}$-free.  
To the contrary assume that for some induced subgraph $H$ of $G$ the complement graph $\overline{H}$ is not colorable with $q$ colors but has no cliques of $q+1$ vertices. 
The last condition tells us that $V(H)\subseteq V(G)\setminus\{u,w\}$ for some $u\in U$ and $w\in W$. 
By \Cref{EB}\eqref{EB2} we see that $\overline{H}$ is colorable with $q$ colors, a contradiction. 
Thus, $I(\sigma_q\Delta(G))$ becomes the desired complete intersection ideal. 
Moreover, by definition $G$ carries at least one perfect matching so that $V(G)$ is a union of $q+1$ edges, which implies that $I(\sigma_{q+1}\Delta(G))=0$.

Let us say that there are only two monomials in $I(\sigma_q\Delta(G))_{q+1}$ and that they consist of distinct variables. Note that the two monomials give a balanced bipartition $V(G)=U\sqcup W$ of $G$ in view of \Cref{Gen}. 
Then it is easy to check \Cref{EB}\eqref{EB1} since $U$ and $W$ are the only sets of $q+1$ vertices with no edges in $G$, hence $G$ is elementary bipartite.
\end{proof}

\begin{Prop}\label{F2toF1}
Let $q\geq 1$ be an integer. 
Then if $G\in\mathcal F_{q,2}$, then $G/e\in\mathcal F_{q,1}$ for every edge $e$ of $G$ not lying in any induced cycle $C_4$ of $G$. 
\end{Prop}

\begin{proof}
We may assume that $G$ has no isolated vertices. Observe that $\sigma_q\Delta(G)/e$ is determined by a single nonface of $2q+1$ vertices. 
Since $\sigma_q\Delta(G)/e\subseteq\sigma_q\Delta(G/e)$, it is enough to show that $\sigma_q\Delta(G/e)$ is not a simplex. Indeed, it is the case because $\Delta(G/e)$ consists of only vertices and edges due to the $C_3$-freeness of $G$ and the choice of $e$.
\end{proof} 

\begin{Rmk}
In the proposition above the edge contracted must avoid induced cycles $C_4$. For instance, consider the following graph.
$$
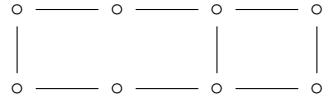
\begin{tikzcd}
    \circ \ar[d,no head] \ar[r,no head] & \circ \ar[r,no head] & \circ \ar[r,no head] \ar[d,no head] & \circ \ar[d,no head] \\ 
    \circ \ar[r,no head] & \circ \ar[r,no head] & \circ \ar[r,no head] & \circ
\end{tikzcd}
$$
It has an induced cycle $H=C_4$, and if $e$ is any edge of $H$, then $G/e$ has its vertices covered by $3$ cliques including $V(H)$. 
\end{Rmk}

Now we establish a lemma that is used in our main theorem. 
\begin{Lem}\label{bipartite_edge_contraction}
Let $G$ be a graph on $r+1$ vertices. 
Suppose that every edge contractions of $G$ is bipartite.
Then $G$ is either a forest or $C_n\sqcup\overline{K_{r-n+1}}$ for some odd integer $3\leq n\leq r+1$.
\end{Lem}

\begin{proof}
Assume that $G$ is not a forest.
Then $G$ has an induced cycle $C_n$ for some $3\leq n\leq r+1$. 
If $n$ were even, and if $e$ were an edge of $C_n$, then $G/e$ would have an (induced) odd cycle $C_{n-1}$, a contraction.
On the other hand, if $n$ were odd, and if $G$ had an edge $e$ other than those of $C_n$, then $C_n$ would still form an odd cycle of $G/e$, a contradiction. Consequently, the only possibility is that $G=C_n\sqcup\overline{K_{r-n+1}}$ with $3\leq n\leq r+1$ odd.
\end{proof}

\begin{Ex}
We illustrate the families $\mathcal F_{q,1}$ and $\mathcal F_{q,2}$ for $q$ small in \Cref{tab:Ftable}.

\begin{table}[h!]
    \centering
    \begin{tabular}{c|lll}
        $q$                 & $1$           & $2$                                                   & $\geq 3$ \\ \hline
        \\[-1em]
        $\mathcal F_{q,1}$  & $\emptyset$   & $\overline{C_n}$, $n\geq 5$ odd (\Cref{Gen_Secant})   & cf.\ \Cref{F2toF1} \\
        $\mathcal F_{q,2}$  & $C_4$         & $C_4\sqcup\overline{K_1}$ and graphs in \Cref{fig:EB6}& See \Cref{EBstructure} \\
    \end{tabular}
    \caption{Families $\mathcal F_{q,1}$ and $\mathcal F_{q,2}$ for $q$ small}
    \label{tab:Ftable}
\end{table}

It would not be easy to describe $\mathcal F_{q,1}$ for the integers $q\geq 3$. 
However computationally it is rather simple to find all the members of $\mathcal F_{q,2}$ for the integers $q\geq 3$, for the structure of elementary bipartite graphs is well understood by \Cref{EBstructure}. 
For example, $\mathcal F_{3,2}$ is the set of the following graphs.
\begin{enumerate}
    \item $C_4\sqcup\overline{K_2}$,
    \item $H\sqcup\overline{K_1}$ for graphs $H$ in \Cref{fig:EB6}, and
    \item $26$ nonisomorphic elementary bipartite graphs on $8$ vertices.
\end{enumerate}
\end{Ex}

Now we summarize algebraic feature of graphs in $\mathcal F_{q,1}$ and $\mathcal F_{q,2}$ and cycles $C_{2q+c}$ with $c\geq 3$ in terms of Betti tables. To this end we need a lemma. 
Recall that for simplicial complexes $\Delta_1$ and  $\Delta_2$ on $V_1$ and $V_2$, respectively, the disjoint union $\Delta_1\sqcup\Delta_2$ is a simplicial complex on $V_1\sqcup V_2$ defined by
$$
\Delta_1\sqcup\Delta_2=\{F\subseteq V_1\sqcup V_2:\textup{either }F\in\Delta_1\textup{ or }F\in\Delta_2\}.
$$
Then if for each $i=1,2$ we write $\Bbbk[x_{i,0}, x_{i,1}, \ldots,x_{i,r_i}]$ for the polynomial ring assigned to the vertex set of $\Delta_i$, then the vertex set of $\Delta_1\sqcup\Delta_2$ carries the polynomial ring $\Bbbk[x_{1,0},x_{1,1},\ldots,x_{1,r_1},x_{2,0},x_{2,1},\ldots,x_{2,r_2}]$ in $r_1+r_2+2$ variables for itself.

\begin{Lem}\label{Isol}
Let $\Delta_1$ and $\Delta_2$ be simplicial complexes on vertex sets $V_1$ and $V_2$, respectively.
Then we obtain
$$I(\sigma_q(\Delta_1\sqcup\Delta_2))= \sum_{j = -1}^qI(\sigma_{q-j-1}\Delta_1)\cdot I(\sigma_j\Delta_2)$$
where for convenience we set $I(\sigma_{-1}\Delta)$ to be the unit ideal and $I(\sigma_0\Delta)$ to be the maximal homogeneous ideal.
\end{Lem}
\begin{proof}
One can see that it is enough to show the equality
    $$\sigma_q(\Delta_1\sqcup\Delta_2)=\bigcap_{j=-1}^q(\sigma_{q-j-1}\Delta_1\ast 2^{V_2})\cup(2^{V_1}\ast\sigma_j\Delta_2),$$
    where $\sigma_{-1}\Delta$ and $\sigma_0\Delta$ are naturally understood by $I(\sigma_{-1}\Delta)$ and $I(\sigma_0\Delta)$, respectively.

Let $F$ be an arbitrary face of $\sigma_q(\Delta_1\sqcup\Delta_2)$. By definition every face of $\Delta_1\sqcup\Delta_2$ lies in either $\Delta_1$ or $\Delta_2$, which implies that 
$$F=F_1\cup F_2$$ 
for some $F_1\in\sigma_{q-k}\Delta_1$ and $F_2\in\sigma_k\Delta_2$ with $0\leq k\leq q$. If $-1\leq j\leq k-1$, then $F_1\in\sigma_{q-j-1}\Delta_1$, and if $k\leq j\leq q$, then $F_2\in\sigma_j\Delta_2$. 

On the other hand, choose any face $F$ of the simplicial complex on the right-hand side. Put 
$$
k=\min\{-1\leq j\leq q:F\in 2^{V_1}\ast\sigma_j\Delta_2\}
$$
so that $F\in(\sigma_{q-k}\Delta_1\ast 2^{V_2})\cap(2^{V_1}\ast\sigma_k\Delta_2)\subset\sigma_q(\Delta_1\sqcup\Delta_2)$.
\end{proof}

The desired Betti tables are computed as follows.

\begin{Prop}\label{FBettiTable}
Let $G$ be a graph and $q\geq 1$ be an integer. Then the following hold.
\begin{enumerate}
    \item\label{FBettiTable1} If $G\in\mathcal F_{q,1}$, then the Betti table of $S(\sigma_q\Delta(G))$ is given by \textup{\Cref{F1BettiTable}}, where $r+1>q+1$ is the number of vertices in $G$.
    \item\label{FBettiTable2} If $G\in\mathcal F_{q,2}$, then the Betti table of $S(\sigma_q\Delta(G))$ is given by \textup{\Cref{F2BettiTable}}, where $0\leq s<q$ is the number of isolated vertices in $G$.
    \item\label{FBettiTable3} If $G=C_{2q+c}$ for an integer $c\geq 3$, then the Betti table of $S(\sigma_q\Delta(G))$ is given by \textup{\Cref{DelPezzo}}.
\end{enumerate}
\end{Prop}
\begin{table}[h!]
\begin{minipage}[b]{.45\textwidth}
$$
\begin{array}{c|ccc}
     & 0 & 1 \\ \hline
    0 & 1 & . \\
    r & . & 1
\end{array}
$$
    \caption{$G\in\mathcal F_{q,1}$}
    \label{F1BettiTable}
\end{minipage}
\begin{minipage}[b]{.45\textwidth}
$$
\begin{array}{c|ccc}
     & 0 & 1 & 2 \\ \hline
    0 & 1 & . & . \\
    q & . & 2 & . \\
    2q-s & . & . & 1
\end{array}
$$
    \caption{$G\in\mathcal F_{q,2}$}
    \label{F2BettiTable}
\end{minipage}
\end{table}

\begin{proof}
    \eqref{FBettiTable1} It is immediate from \Cref{Gen}.
    
    \eqref{FBettiTable2} Since $G \in \mathcal{F}_{q,2}$, we have $G = H \sqcup \overline{K_s}$, where $H$ is an elementary bipartite graph on $2q - 2s + 2$ for $0 \leq s \leq q$. 
    Let $x_0, x_1, \dots, x_{q-s}$ and $y_0, y_1, \dots, y_{q-s}$ be the variables corresponding to the vertices in $U$ and $W$, respectively, where $U$ and $W$ are the bipartition sets of the vertex set of $H$. 
    Let $z_1, z_2, \dots, z_s$ be the variables corresponding to the vertices of $K_s$. 
    Then, the Stanley-Reisner ideal of $\sigma_q(\Delta(G))$ is generated by the two monomials 
    $$ (x_0 x_1 \cdots x_{q-s})(z_1 z_2\cdots z_s) \quad\text{and}\quad (y_0 y_1 \cdots y_{q-s})(z_1 z_2 \cdots z_s)$$
    by \Cref{EBci} and \Cref{Isol}.
    Therefore, we obtain the Betti table as shown in \Cref{F2BettiTable}.

    \eqref{FBettiTable3} By \Cref{thm:pure} it would suffice to show the required vanishing of Betti numbers only. 
    In order to proceed by \Cref{thm:Homology}, take any nonempty subset $W\subseteq V(C_{2q+c})$. 
    We divide into two cases.

    {\it Case}: $W\neq V(C_{2q+c})$.
    The induced subgraph $C_{2q+c}[W]$ becomes isomorphic to that of $P_{2q+c-1}$, and \Cref{cor:PKBetti} says that $\widetilde{H}_{j-1}(\sigma_q\Delta(C_{2q+c})[W],\Bbbk)\neq 0$ only if $|W|\leq c+q-1$ and $j=q$.

    {\it Case}: $W=V(C_{2q+c})$.
    We pick an edge $e$ of $C_{2q+c}$.
    Then by \Cref{cor:EdgePreserve} we have $\widetilde{H}_{j-1}(\sigma_q\Delta(C_{2q+c}),\Bbbk)\cong\widetilde{H}_{j-1}(\sigma_q\Delta(C_{2q+c})/e,\Bbbk)\cong\widetilde{H}_{j-1}(\sigma_q\Delta(C_{2q+c-1}),\Bbbk).$
    Applying induction on the number of vertices we conclude that $\widetilde{H}_{j-1}(\sigma_q\Delta(C_{2q+c}),\Bbbk)\neq 0$ if and only if $j=2q$.
    Note that \eqref{FBettiTable1} or \eqref{FBettiTable2} provides the initial step of induction.
\end{proof}

For the rest of this section we introduce interactions between graphs in the families $\mathcal F_{q,1},\mathcal F_{q,2},\mathcal F_{q+1,1},\mathcal F_{q+1,2}$, and cycles in terms of forbidden induced subgraphs, which is used in \Cref{Sec:Main}. We refer to the following theorems.

\begin{Thm}[Strong perfect graph theorem, \cite{MR2233847}]\label{Perfect}
A graph $G$ is perfect if and only if $G$ is $C_n$- and $\overline{C_n}$-free for all odd integers $n\geq 5$.
\end{Thm}

An equivalent formulation of the theorem above has been suggested by Sturmfels and Sullivant.

\begin{Thm}[{\cite[Proposition 3.4 and Corollary 3.7]{MR2252121}}] \label{Perfect'}
A graph $G$ is not perfect if and only if there is an integer $q\geq 2$ such that $I(\sigma_q\Delta(G))$ is not generated in degree $q+1$. Furthermore, in this case if $q$ is set to be smallest, then either
\begin{enumerate}
    \item $q=2$, and $\beta_{1,j}(\sigma_q\Delta(G))\neq 0$ implies $j\in\{2,4,\ldots\}$, or
    \item $q\geq 3$, and $\beta_{1,j}(\sigma_q\Delta(G))\neq 0$ implies either $j=q$ or $j=2q$.
\end{enumerate}
\end{Thm}

Building upon them we find a criterion for the $\mathcal F_{q+1,1}$-freeness.

\begin{Prop}\label{Forbidden1}
Let $G$ be a graph and $q \geq 1$ be an integer.
Suppose that $G$ is $\mathcal F_{j,1}$-free for all $2\leq j<q$ and that either
\begin{enumerate}
    \item\label{Forbidden1-1} $q=1$, and $G$ is $\mathcal F_{1,2}$-free and $C_5$-free, or
    \item\label{Forbidden1-2} $q\geq 2$, and $G$ is $\mathcal F_{q,1}$-free and $C_{2q+3}$-free. 
\end{enumerate}
Then $G$ is $\mathcal F_{q+1,1}$-free.
\end{Prop}

\begin{proof}
To the contrary assume that $G$ is not $\mathcal F_{q+1,1}$-free. We divide into cases according to the value of $q$.

\eqref{Forbidden1-1} Thanks to \Cref{Gen_Secant} we may assume that $G=\overline{C_n}$ for some odd integer $n\geq 5$. But $\overline{C_5}=C_5$, and $\overline{C_m}$ has an induced cycle $C_4$ for all $m\geq 6$, a contradiction.

\eqref{Forbidden1-2} By \Cref{Perfect'} the graph $G$ is imperfect with $\beta_{1,2q}(\sigma_q\Delta(G))\neq 0$.
Taking an appropriate induced subgraph if necessary we may set $G$ to have exactly $2q+3$ vertices.
Note that $G$ is $C_{2j+1}$-free for all $2\leq j\leq q$ since $C_{2j+1}\in\mathcal F_{j,1}$ for all $j\geq 2$ and that it is also $\overline{C_n}$-free for every odd integer $n\geq 5$ due to \Cref{Gen_Secant}.
Now \Cref{Perfect} implies that $G \cong C_{2q+3}$, a contradiction.
\end{proof}

Let us see the families $\mathcal{F}_{q+1,2}$.

\begin{Prop}\label{Forbidden2}
Let $G$ be a graph and $q\geq 1$ be an integer. Then if $G$ is $\mathcal F_{q,2}$- and $C_{2q+4}$-free, then it is $\mathcal F_{q+1,2}$-free.
\end{Prop}

\begin{proof}
To the contrary suppose that $G$ is not $\mathcal F_{q+1,2}$-free. We may assume that $G=H\sqcup\overline{K_{q-j+1}}$ for some $1\leq j\leq q+1$ and some elementary bipartite graph $H$ on $2j+2$ vertices. If $j\leq q$, then $H\sqcup\overline{K_{q-j}}\in\mathcal F_{q,2}$, a contradiction. So $G$ is elementary bipartite with exactly $2q+4$ vertices. According to \Cref{EBstructure} one can write
$$G=G'\cup P_{2q-2k+4}$$
for an integer $0\leq k\leq q+1$, where $G'$ is an elementary bipartite graph on $2k+2$ vertices, and $P_{2q-2k+4}$ is the added path sharing two vertices with $G'$. Then if $k\geq 1$, then by taking $q-k$ suitable vertices from $P_{2q-2k+4}$ the graph $G$ has
$$G'\sqcup\overline{K_{q-k}}\in\mathcal F_{q,2}$$
as an induced subgraph, a contradiction, and otherwise we obtain $G=C_{2q+4}$, a contradiction again. Consequently $G$ is $\mathcal F_{q+1,2}$-free.
\end{proof}

\section{Property $N_{q+1,p}$ of $q$-secant complexes} \label{Sec:Main}
Eisenbud et al.\ have refined Fröberg's seminal work on the complete classification of $1$-regular complexes by providing a precise characterization of graphs that satisfy property $N_{2,p}$:

\begin{Thm}[{\cite[Theorem 2.1]{MR2188445}}]
    Let $G$ be a graph. Then $S(\Delta(G))$ satisfies property $N_{2,p}$ for an integer $p\geq2$ if and only if every cycle of $G$ with at most $p + 2$ vertices has a chord.
\end{Thm}
We extend this to higher secant complexes.

\begin{proof}[Proof of \Cref{thm:main}]
Given a simplicial complex $\Delta$ on a finite set $V = \{0,1,\ldots, r\}$ suppose that $S(\sigma_q\Delta)$ satisfies property $N_{q+1,1}$. 
Let $G_0$ be the underlying graph of $\Delta$, that is,
$$G_0=\{F\in\Delta:\dim F\leq 1\}.$$
One can see that $\Delta\subseteq\Delta(G_0)$, and so $\sigma_q\Delta\subseteq\sigma_q\Delta(G_0)$. 
By the assumption above \Cref{thm:Prolong} implies the equality $\sigma_q\Delta=\sigma_q\Delta(G_0)$.
Then the ``only if " parts of the statements \eqref{Main1}, \eqref{Main2}, and \eqref{Main3} can be verified using \Cref{FBettiTable}.
We now turn our attention to proving the ``if " parts.

\eqref{Main1} It is straightforward due to \Cref{Gen}.

\eqref{Main2} We prove this by contrapositive.
We assume that $S(\sigma_q\Delta(G))$ satisfies property $N_{q+1,1}$ but does not satisfy property $N_{q+1,2}$.
Note that every relation between monomials is spanned by those of the form
$$
\rel(m_1,m_2):\quad \left(\frac{m_2}{\gcd(m_1,m_2)}\right)m_1-\left(\frac{m_1}{\gcd(m_1,m_2)}\right)m_2=0,
$$
where $m_1$ and $m_2$ are any two monomials of degree $q+1$.
Therefore, for a suitable numbering of the vertices, $I(\sigma_q \Delta(G))$ contains two monomials given by
$$m_1 = (x_0 x_1 \cdots x_j) (x_{2j+2} x_{2j+3} \cdots x_{q+j+1}) \quad\text{and}\quad m_2 = (x_{j+1} x_{j+2} \cdots x_{2j+1})  (x_{2j+2} x_{2j+3} \cdots x_{q+j+1})$$
with $1 \leq j \leq q$.
Assume that the integer $j$ is the minimum among such integers, and replace $G$ with the induced subgraph $G[\{0, 1, \ldots, q+j+1\}]$.
By \Cref{thm:Prolong} the vertices $2j+2, 2j+3, \ldots, q+j+1$ are all isolated in $G$.
Therefore, for the induced subgraph $H=G[\{0, 1, \ldots, 2j+1\}]$, we have 
$$(x_0 x_1 \cdots x_j,x_{j+1} x_{j+2}\cdots x_{2j+1}) \subseteq I(\sigma_j\Delta(H))$$ by \Cref{Isol}.

We claim that there are no other monomial generators in $I(\sigma_j \Delta(H))$.
For the sake of contradiction suppose that there exists another monomial generator $m_3'$ of degree $j+1$ in $I(\sigma_j \Delta(H))$. 
We set $m_3 = m_3' \cdot (x_{2j+2} x_{2j+3} \cdots x_{q+j+1})$.
Then by \Cref{Isol} we have $m_3 \in I(\sigma_q \Delta(G))$. 
Since $\rel(m_1, m_3)$ and $\rel(m_3, m_2)$ span $\rel(m_1, m_2)$, one of them leads to a contradiction to the minimality of $j$. 

Thus, 
$$(x_0 x_1 \cdots x_j,x_{j+1} x_{j+2} \cdots x_{2j+1}) = I(\sigma_j \Delta(H)),$$
and we conclude that $G = H \sqcup \overline{K_{q-j}} \in \mathcal{F}_{q,2}$ since $H$ is elementary bipartite by \Cref{EBci}.

\eqref{Main3} We prove this by contrapositive.  
Assume that $S(\sigma_q\Delta(G))$ satisfies property $N_{q+1,p-1}$ but not property $N_{q+1,p}$ for $p \geq 3$.

First, suppose $p = 3$.  
Let $j$ be the smallest positive integer such that $\beta_{3,q+j}(\sigma_q\Delta(G)) \neq 0$.  
Let $G'=G[W]$ be an induced subgraph with $|W| = q+j+3$ such that
$$\widetilde{H}_{q+j-1}(\sigma_q \Delta(G'),\Bbbk) = \widetilde{H}_{q+j-1}(\sigma_q \Delta(G)[W],\Bbbk) \neq 0.$$
By \Cref{cor:EdgePreserve} for any edge $e$ of $G'$ the Stanley-Reisner ring $S(\sigma_q\Delta(G')/e)$ satisfies property $N_{q+1,1}$, but
\begin{equation}\label{Betti_special}
\beta_{2,q+k}(\sigma_q\Delta(G')/e) =
\begin{cases}
0 & \text{if } 1 \leq k \leq j - 1, \\
\dim_\Bbbk \widetilde{H}_{q+j-1}(\sigma_q\Delta(G')/e, \Bbbk) \neq 0 & \text{if } k = j.
\end{cases}
\end{equation}
Note that by \Cref{thm:Prolong} we have $\sigma_q\Delta(G')/e = \sigma_q\Delta(G'/e)$ since $I(\sigma_q\Delta(G')/e)$ is generated in degree $q+1$.  
Using \eqref{Betti_special} and \Cref{thm:main}\eqref{Main2}, we conclude that the edge contraction $G'/e$ is isomorphic to $H \sqcup \overline{K_{q-j}}$, where $H$ is an elementary bipartite graph with exactly $2j+2$ vertices.  
Since $e$ was chosen arbitrarily, \Cref{bipartite_edge_contraction} implies $G' = C_{2j+3} \sqcup \overline{K_{q-j}}$.  

However, if $j < q$, then by \Cref{Isol} and \Cref{FBettiTable}\eqref{FBettiTable1} the Stanley-Reisner ideal $I(\sigma_q\Delta(G'))$ would not be generated in degree $q+1$ which contradicts that $S(\sigma_q\Delta(G')/e)$ satisfies property $N_{q+1,1}$.  
Hence, we conclude that $G$ contains an induced subgraph isomorphic to $C_{2q+3}$.

For $p \geq 4$ it follows that an edge contraction $G/e$ has an induced cycle $C_{2q+p-1}$.  
However, $G$ is $C_{2q+p-1}$-free.
In conclusion, $G$ contains an induced cycle $C_{2q+p}$. 
\end{proof}

By further exploring the behavior of Betti tables under edge contractions, we may derive additional results on the vanishing patterns of graded Betti numbers.

\begin{proof}[Proof of \Cref{MainCor}]
For any integer $q\geq 1$, if $S(\sigma_q\Delta)$ satisfies property $N_{q+1,1}$, then there is a graph $G$ on $V(\Delta)$ such that $\sigma_q\Delta(G)=\sigma_q \Delta$.
Hence we can replace $\Delta$ with $\Delta(G)$ due to \Cref{Eventually}.

By \Cref{thm:main}\eqref{Main3} we know that $G$ is $\mathcal{F}_{j,1}$-free for all $2 \leq j < q$ and also $\mathcal{F}_{q,1}$-, $\mathcal{F}_{q,2}$-, and $C_{2q+i}$-free for all $3 \leq i \leq p$.  
Therefore, the graph $G$ is $\mathcal{F}_{q+1,1}$-free by \Cref{Forbidden1}.
By \Cref{Forbidden2} it is $\mathcal{F}_{q+1,2}$-free when $p \geq 4$ and it is already $C_{2(q+1)+i}$-free for all $3 \leq i \leq p-2$ provided that $p \geq 5$.  
This completes the proof by \Cref{thm:main}.
\end{proof}

In addition, the family of $q$-secant chordal graphs is closed under edge contraction. 
More generally, edge contractions shift, by one, the index $p$ of property $N_{q+1, p}$ satisfied.

\begin{proof}[Proof of \Cref{InnerProj}]
Suppose that $S(\sigma_q\Delta)$ satisfies property $N_{q+1,p}$ for an integer $p\geq 1$, and let $e$ be an edge of $\Delta$. For the case $p=1$ there is nothing to prove, hence we set $p\geq 2$. In order to find whether $S(\sigma_q\Delta/e)$ satisfies property $N_{q+1,p-1}$ take an arbitrary set $W$ of $i+j$ vertices of $\Delta/e$ with $i\leq p-1$ and $j\geq q+1$. Write $v_e$ for the new vertex of $\Delta/e$ obtained by contracting $e$. If $v_e\not\in W$, then $\sigma_q\Delta/e[W]=\sigma_q\Delta[W]$, and so their reduced homology groups are trivially the same. If $v_e\in W$, then the set $\widetilde{W}:=(W\setminus v_e)\cup e$ of vertices of $\Delta$ has $\widetilde{H}_{j-1}(\sigma_q\Delta[\widetilde{W}],\Bbbk)\cong\widetilde{H}_{j-1}(\sigma_q\Delta/e[W],\Bbbk)$ by \Cref{cor:EdgePreserve}. By our assumption we have $\widetilde{H}_{j-1}(\sigma_q\Delta/e[W],\Bbbk)=0$ and thus conclude that $S(\sigma_q\Delta/e)$ satisfies property $N_{q+1,p-1}$.

For the special case of $q$-secant chordal graphs, \Cref{thm:main} says that $I(\sigma_q\Delta(G))$ has $(q+1)$-linear resolution, and as above \Cref{cor:EdgePreserve} and \Cref{thm:Prolong} tell us that so does $I(\sigma_q\Delta(G/e))$. Applying \Cref{thm:main} again we see that $G/e$ is $q$-secant chordal.
\end{proof}

Through the study of graded Betti numbers and property $N_{q+1, p}$ of $q$-secant complexes, we may classify simplicial complexes whose $q$-secant complexes are del Pezzo (see \Cref{Def:delPezzo}).

\begin{proof}[Proof of \Cref{MainDelPezzo}]
    It immediately follows from \Cref{thm:main} and \Cref{FBettiTable}.
\end{proof}

\bigskip

\section{Numerical invariants and Cohen-Macaulayness}\label{Sec:CM}
In this section as a complementary observation we interpret the value of $\beta_{p,q}(\sigma_q\Delta(G))$ for a graph $G$ in the language of the connectivity of induced subgraphs when $S(\sigma_q\Delta(G))$ satisfies property $N_{q+1,p-1}$. 
This part is also devoted to computing the projective dimension of $S(\sigma_q\Delta(G))$ when $G$ is $q$-secant chordal. 

We specify collections of induced subgraphs that provide numerical building blocks for our purpose.

\begin{Def}
Let $G$ be a graph. For integers $n\geq s+1\geq 1$ we define $\mathcal H_n^s(G)$ to be the set of induced subgraphs $H$ on $n$ vertices of $G$ such that $H$ has exactly $s+1$ connected components.
\end{Def}

\Cref{thm:numerical} presents its significance. We prove it as follows:

\begin{proof}[Proof of \Cref{thm:numerical}]
We proceed by \Cref{thm:Homology}.
Let $H$ be an induced subgraph on $p+q$ vertices of $G$ such that $\widetilde{H}_{q-1}(\sigma_q\Delta(H),\Bbbk)\neq 0$. Take $s+1$ to be the number of connected components of $H$. Note that $p+q-s-1$ successive edge contractions of $H$ are possible, resulting in $\overline{K_{s+1}}$. Then presuming that $s\geq q$ we have
$$
\dim_\Bbbk\widetilde{H}_{q-1}(\sigma_q\Delta(H),\Bbbk)=\dim_\Bbbk\widetilde{H}_{q-1}(\sigma_q\Delta(\overline{K_{s+1}}),\Bbbk)=\binom{s}{q}
$$
by \Cref{cor:EdgePreserve} and \Cref{cor:PKBetti}.

We claim that $s\geq q$. To the contrary say that $s\leq q-1$. Then taking $p-1$ successive edge contractions of $H$ if necessary, we restrict ourselves to the case $p=1$, that is, $|V(H)|=q+1$. If $H$ had an edge, then $\sigma_q\Delta(H)$ would be a simplex so that $\widetilde{H}_{q-1}(\sigma_q\Delta(H),\Bbbk)=0$, a contradiction. Therefore, $H=\overline{K_{q+1}}$, and so $s=q$, a contradiction. Our declaration has been confirmed.
\end{proof}

\begin{Rmk}
Let $G$ be a graph and $n\geq s+1\geq 1$ be integers. Assume that $S(\sigma_q\Delta(G))$ satisfies property $N_{q+1,n-q-1}$ for all $s\leq q\leq n-1$. Then using a well-known inversion formula we get
$$
|\mathcal H_n^s(G)|=\sum_{q=s}^{n-1}(-1)^{q-s}\binom{q}{s}\beta_{n-q,q}(\sigma_q\Delta(G)).
$$
\end{Rmk}

As a consequence of \Cref{thm:numerical} for a $q$-secant chordal graph $G$ the sets $\mathcal H^s_n(G)$ can be shown to include information on the projective dimension of $S(\sigma_q\Delta(G))$.

\begin{Cor}\label{cor:pd}
Let $G$ be a $q$-secant chordal graph with $\sigma_q\Delta(G)$ not a simplex for an integer $q\geq1$. Then the projective dimension $\pd S(\sigma_q\Delta(G))$ is computed as
$$
\max\{|V(H)|-q:H\textup{ has at least }q+1\textup{ connected components}\},
$$
where $H$ stands for an induced subgraph of $G$.
\end{Cor}

Now let us discuss how projective dimension changes under edge contractions. In the spirit of \Cref{cor:EdgePreserve} one can observe that if $\Delta$ is a simplicial complex such that $S(\sigma_q\Delta)$ satisfies property $N_{q+1,1}$, then we have $\pd S(\sigma_q\Delta/e)\leq\pd S(\sigma_q\Delta)$ for every edge $e$ of $\Delta$. But for the sake of consistency we restrict ourselves to clique complexes of $q$-secant chordal graphs.

\begin{Cor}\label{cor:prdim}
Let $G$ be a $q$-secant chordal graph for an integer $q\geq1$ and $e$ be an edge of $G$. Then we obtain 
$$
\pd S(\sigma_q\Delta(G/e))\leq\pd S(\sigma_q\Delta(G))
$$ 
with equality if and only if an induced subgraph $H\in\mathcal H_{p+q}^s(G)$ shares no vertices with $e$ for the integer $p=\pd S(\sigma_q\Delta(G))$ and any integer $q\leq s\leq p+q-1$.
\end{Cor} 

\begin{proof}
Observe that
$$
\mathcal H_n^s(G/e)=\{H\in\mathcal H_n^s(G):V(H)\cap e=\emptyset\}\cup\{H/e:H\in\mathcal H_{n+1}^s(G)\textup{ with }E(H)\ni e\}
$$
for all integers $n\geq s+1\geq 1$. Due to \Cref{thm:numerical} we are done.
\end{proof}

\Cref{ex:projdim} illustrates two cases: one where the projective dimension decreases and the other where it remains the same, depending on which edge is chosen for the contraction.

\begin{Ex}\label{ex:projdim}
Let $\Delta$ be the clique complex of the graph $G$ in \Cref{fig:3regbutminimalsecants} of \Cref{ex:sunlet}, and take $e_1$ (resp.\ $e_2$) to be an edge inside (resp.\ outside) the induced cycle $C_4$. 
Then $\pd S(\sigma_2 \Delta) = 4$, but we obtain $\pd S(\sigma_2 \Delta/e_1) = 4$ and $\pd S(\sigma_2 \Delta/e_2) = 3$.

\begin{table}[h!]
    \centering
    \begin{tabular}{c|c c c c c}
         $j \setminus i$  & 0 & 1 & 2 & 3 & 4\\
       \hline
      0 & 1 & . & . & . & .\\
      1 & . & . & . & . & .\\
      2 & . & 20 & 45 & 36 & 10
     \end{tabular}
    \caption{Betti table of $\sigma_2\Delta$}
    \label{tab:prodimSec2Delta}
\end{table}
\begin{table}[h!]
    \begin{subtable}[h]{0.45\textwidth}
    \centering
        \begin{tabular}{c|c c c c c}
          $j \setminus i$ & 0 & 1 & 2 & 3 & 4\\
        \hline
      0 & 1 & . & . & . & .\\
      1 & . & . & . & . & .\\
      2 & . & 11 & 18 & 9 & 1 \\ 
    \end{tabular} 
    \caption{Betti table of $\sigma_2 \Delta/e_1$}
    \end{subtable}
    \begin{subtable}[h]{0.45\textwidth}
    \centering
        \begin{tabular}{c|c c c c}
        $j \setminus i$ & 0 & 1 & 2 & 3\\
       \hline
      0 & 1 & . & . & .\\
      1 & . & . & . & .\\
      2 & . & 10 & 15 & 6
       \end{tabular}
    \caption{Betti table of $\sigma_2 \Delta/e_2$}
    \end{subtable}
    \caption{Betti tables of $\sigma_2\Delta/e_1$ and $\sigma_2\Delta/e_2$}
    \label{tab:prodimSec2Delta12}
\end{table}
\end{Ex}

 As an application of \Cref{cor:pd} we discuss \emph{Cohen-Macaulay} simplicial complexes.  
A simplicial complex $\Delta$ is called \emph{Cohen-Macaulay} if its Stanley-Reisner ring $S(\Delta)$ is Cohen-Macaulay, meaning that
$$\pd S(\Delta) = \codim\Delta.$$

We recall that any Cohen-Macaulay complex is pure, meaning that all of its facets have the same dimension.
Moreover, any Cohen-Macaulay complex of positive dimension is connected.
In addition, if a pure complex $\Delta$ admits a \emph{shelling order}, then it is Cohen-Macaulay; see \cite[Section 8]{MR2724673} for further details.

We focus on \emph{forests} $G$ and characterize those for which $\sigma_q\Delta(G)$ are Cohen-Macaulay for a given integer $q\geq 1$.
Here, we define a \emph{forest} as a disjoint union of \emph{trees}, that is, connected graphs having no cycles. Recall that a vertex $v$ in a forest $G$ is called
\begin{enumerate}
    \item a \emph{leaf} if $\deg_G v=1$,
    \item an \emph{internal vertex} if $\deg_G v\geq 2$, and
    \item a \emph{branch vertex} if $\deg_G v\geq 3$.
\end{enumerate}

We remark some facts about forests.
\begin{Rmk}\label{Rmk:forForests}
Let $G$ be a forest and $q\geq 1$ be an integer.
\begin{enumerate}
    \item $\Delta(G)$ has dimension at most one because $G$ is $C_3$-free.
    \item $\dim\sigma_q\Delta(G)\leq 2q-1$, and the equality holds if and only if the matching number is $\nu(G)\geq q$.
    \item $G$ is $q$-secant chordal by \Cref{thm:2Regular}.
    \item\label{Rmk:ConnforForests} $\Delta(G)$ is Cohen-Macaulay if and only if $G$ is a tree or $\overline{K_{r+1}}$, where $r = |V(G)| - 1$. 
    Indeed, if $G = \overline{K_{r+1}}$, then $\Delta(G)$ is Cohen-Macaulay by \Cref{cor:PKBetti}.
    If $G$ is a tree with at least one edge, then $\pd S(\Delta(G)) = |V(G)|-2 = \codim \Delta(G)$ by \Cref{cor:pd}.
    (Consider an internal vertex of $G$ unless $G=P_2$.) The converse is straightforward since $\Delta(G)$ must be connected.
\end{enumerate}
\end{Rmk}

Motivated by a property of projective curves: for a nondegenerate projective variety $X\subseteq\mathbb P^r$ over an algebraically closed field, if $X$ is a curve, then $\dim \sigma_q X = 2q - 1$ unless $\sigma_q X = \mathbb{P}^r$ (see \cite{palatini1906sulle}),
we assume that the forest $G$ has $\nu(G) \geq q$ for the remainder of this section.

\Cref{cor:pd} yields a criterion for Cohen-Macaulayness of higher secant complexes for such forests. Notice that if we delete an internal vertex of a forest, then the number of connected components strictly increases.

\begin{Cor}\label{Cor:Criterion}
Let $G$ be a forest.
    Suppose that $\nu(G) \geq q \geq 1$ and that $\sigma_q \Delta(G)$ is not a simplex.  
    Then $\sigma_q \Delta(G)$ is not Cohen-Macaulay if and only if there exists a set of $q-1$ vertices in $G$ whose deletion results in at least $q+1$ connected components.
\end{Cor}

\begin{proof}  
    Note that $\codim \sigma_q \Delta(G) = |V(G)| - 2q$.  
    In light of \Cref{cor:pd} we have $\pd S(\sigma_q \Delta(G)) \geq |V(G)| - 2q + 1$ if and only if $G$ contains a set of $q-1$ vertices whose deletion induces at least $q+1$ connected components.
\end{proof}

Now we show that under a reasonable assumption the path graphs are the only graphs among forests whose higher secant complexes are Cohen-Macaulay.
We have seen that $\sigma_q \Delta(P_{r+1})$ is Cohen-Macaulay for any integer $q\ge 1$ in \Cref{cor:PKBetti}.

We begin with a lemma that generalizes part of \Cref{Rmk:forForests}\eqref{Rmk:ConnforForests}. 

\begin{Lem}\label{lem:components}
For a forest $G$ with $\nu(G) \ge q\geq 2$ the $q$-secant complex $\sigma_q\Delta(G)$ is not Cohen-Macaulay whenever $G$ is disconnected.
\end{Lem}

\begin{proof}
Suppose that $G = G_1 \sqcup G_2$. 
For a given matching $M$ of size $q$ in $G$, let $q_1$ and $q_2$ denote the number of edges in $M$ that belong to $G_1$ and $G_2$, respectively, with $q_1 + q_2 = q$.
If $G_1$ or $G_2$ has no edges, then $\sigma_q \Delta(G)$ is not pure.
Therefore, one can assume that both $q_1, q_2 \geq 1$.  
Moreover, since $\sigma_q \Delta(G)$ is not a simplex, we may suppose without loss of generality that $\sigma_{q_1} \Delta(G_1)$ is not a simplex.  
Also, we remark that $\sigma_{q_2-1} \Delta(G_2)$ is also not a simplex when $q_2 \geq 2$.

By \Cref{thm:numerical} or \Cref{cor:pd} one can remove some subset $W_1$ (resp.\ $W_2$) of $q_1$ (resp.\ $q_2-1$) vertices from $G_1$ (resp.\ $G_2$) in order to obtain at least $q_1 + 1$ (resp.\ $q_2$) connected components.
(The case $q_2=1$ also works.)
Then for the subset $W := W_1 \sqcup W_2$ of $q-1$ vertices, its deletion from $G$ has $q + 1$ or more connected components.
Thus, \Cref{Cor:Criterion} concludes that $\sigma_q\Delta(G)$ is not Cohen-Macaulay.
\end{proof}

In what follows, for a matching $M$ we denote by $V(M)$ the set of endpoints of edges in $M$.

\begin{Lem}\label{Lem:UsefulMatching}
    Let $T$ be a tree such that $\nu(T)\geq q\geq 1$. If $V(T)\geq 2q+2$, then there is a matching $M$ of size $q$ such that $V(T)\setminus V(M)$ contains two different leaves of $T$.
\end{Lem}

\begin{proof}
Let $(M,\{v_1,v_2\})$ be a pair of a size $q$ matching $M$ of $T$ and a subset $\{v_1,v_2\}$ of two distinct vertices in $V(T)\setminus V(M)$ such that the path $P$ connecting $v_1$ and $v_2$ has the maximum distance among such pairs. We claim that $v_1$ and $v_2$ are leaves of $T$.

To the contrary suppose that $v_1$ is not a leaf, which means that there is a vertex $w\in V(T)\setminus V(P)$ adjacent to $v_1$. We divide into two cases according to the location of $w$. If $w\not\in V(M)$, then the path between $w$ and $v_2$ is longer than $P$, a contradiction to the maximality. 
If $w\in V(M)$ with $\{w,u\}\in M$, then replace $\{w,u\}$ with $\{w,v_1\}$ in $M$, and then replace $v_1$ with $u$ for the two-vertex subset, a contraction once again.
\end{proof}

\begin{proof}[Proof of \Cref{thm:CMtrees}]
The ``if'' part immediately follows by \Cref{cor:PKBetti}.
We prove the converse, saying that $\sigma_q\Delta(G)$ is Cohen-Macaulay. 
By \Cref{lem:components} the forest $G$ must be a tree $T$, and by \Cref{Lem:UsefulMatching} it carries a matching $M$ of size $q$ and a pair $\{\ell_1,\ell_2\}$ of leaves of $T$ in $V(T)\setminus V(M)$. Let $P$ be the path between $\ell_1$ and $\ell_2$ in $T$. It is enough to show that $V(P)$ contains no branch vertices of $T$. 

By way of contradiction, based on \Cref{Cor:Criterion}, assume that a branch vertex $b$ of $T$ lies in $P$. Grouping connected components of $T-b$, we write 
$$
T-b=H_1\sqcup H_2\sqcup H_3,
$$
where $\ell_1\in V(H_1)$, $\ell_2\in V(H_2)$, and $V(H_3)\neq\emptyset$. Then $H_1\sqcup H_2\sqcup H_3$ allows a size $q-1$ matching $M'$ contained in $M$ so that for each $i=1,2,3$ the subgraph $H_i$ has a matching $M'_i$ of size $q_i$ together with $q_1+q_2+q_3=q-1$.

Furthermore, for each $i\in\{1,2\}$ and $1\leq j_i\leq q_i$, the $j_i$-secant complex $\sigma_{j_i}\Delta(H_i)$ is not a simplex due to $\ell_i$, hence for every $0\leq j_i\leq q_i$ after deleting a subset $W_i\subseteq V(H_i)$ of $j_i$ vertices from $H_i$, we obtain $j_i+1$ or more connected components. Now consider two cases: $q_3=0$ or $q_3\geq 1$.

Suppose that $q_3=0$. We may assume that $q_2\geq 1$. Put 
$$
W=\{b\}\sqcup W_1\sqcup W_2
$$ 
after setting $j_1=q_1$ and $j_2=q_2-1$. Then $|W|=q-1$, and by removing $W$ from $T$ at least $q+1$ connected components remain.

Now let us see the case $q_3\geq 1$. Note that $V(H_3)$ admits a subset $W_3$ of $q_3-1$ vertices whose deletion from $H_3$ produces $q_3$ or more connected components. Take 
$$
W=\{b\}\sqcup W_1\sqcup W_2\sqcup W_3
$$
together with $j_1=q_1$ and $j_2=q_2$. As above $T$ becomes a union of at least $q+1$ connected components when we delete $W$ from it.
\end{proof}

The condition $V(G) \geq 2q+2$ in \Cref{thm:CMtrees} is necessary.
Indeed, there exist trees $T$ on $2q+1$ vertices for which $\sigma_q\Delta(T)$ is Cohen-Macaulay even though $T$ is not a path.

\begin{Prop}
Let $q\geq 1$ be an integer, $T_0$ be a tree on $q+1$ vertices, and $T$ be the tree obtained from $T_0$ by subdividing every edge of $T_0$ with one new vertex. Then $\sigma_q\Delta(T)$ is Cohen-Macaulay, but $\nu(T)=q$.
\end{Prop}

\begin{proof}
To show that $\sigma_q\Delta(T)$ is Cohen-Macaulay, we claim that $V(T_0)\subset V(T)$ induces the only subgraph on $q+1$ vertices with no edges in $T$. 
To this end we use induction on $q\geq 1$. The initial step $q=1$ is trivial, hence we set $q\geq 2$. Let $\ell$ be a leaf of $T_0$, and take $T_0'=T_0-\ell$ to be its vertex deletion together with the tree $T'$ constructed in the same way for $T_0'$. By the induction hypothesis $V(T_0')\subset V(T')$ gives the unique induced subgraph on $q$ vertices without edges in $T'$. Then in consideration of the deletion of $\ell$ and the vertex adjacent to $\ell$ from $T$, obviously if $H$ is an induced subgraph on $q+1$ vertices with no edges in $T$, then $V(H)=V(T_0')\sqcup\{\ell\}=V(T_0)$, and vice versa.

The part $\nu(T)=q$ is also proved by the same induction argument. Let $v\in V(T)$ be the vertex adjacent to $\ell$. Its removal from $T$ gives a matching of size $\nu(T)-1$ in $T'$. So $\nu(T)-1\leq\nu(T')=q-1$. On the other hand, adding the edge $\{v,\ell\}$ produces a matching of size $\nu(T')+1$ in $T$. Therefore, the induction process says that $\nu(T)=q$.
\end{proof}

\begin{Rmk}
A similar construction yields trees $\widehat{T}$ on $2q+1$ vertices such that $\sigma_q\Delta(\widehat{T})$ is not Cohen-Macaulay, and $\nu(\widehat{T})=q$.

For an integer $q\geq 2$ take $T_0'$ to be a tree on $q$ vertices, and let $T'$ be the subdivision of $T_0'$ as above.
Introduce a new path $P_3$, and identify one of its endpoints with a vertex of $T'$ that does not belong to $T_0'$. Then the formed tree $\widehat{T}$ satisfies $\nu(\widehat{T})=q$ for the same reason as above. However, $\sigma_q\Delta(\widehat{T})$ is not Cohen-Macaulay. Indeed, $T$ has (exactly) two induced subgraphs on $q+1$ vertices with no edges. They contain the only induced subgraph on $q$ vertices without edges in $T'$ and depend on the choice of an unidentified vertex in $P_3$.
\end{Rmk}

\bigskip

\section{Problems}\label{Sec:Problems}
This section lists problems of our interest.

One would ask whether the assumption that $S(\sigma_j\Delta)$ satisfies property $N_{j+1,1}$ for every $2\leq j<q$ is relevant in \Cref{MainCor}. It can be rephrased in the purely graph-theoretic language as follows.

\begin{Q}\label{Quest:qSecChordal}
Let $G$ be a graph and $q\geq 3$ be an integer. If $G$ is $\mathcal F_{q,1}$-, $\mathcal F_{q,2}$-, and $C_{2q+3}$-free, then is it $\mathcal F_{q+1,1}$-free?
\end{Q}

\begin{Rmk}
A searching with Macaulay2 \cite{M2} witnesses that for the case $q=3$ there are no counterexamples to the statement in \Cref{Quest:qSecChordal} up to $9$ vertices.
\end{Rmk}

The converse of the statement in \Cref{Quest:qSecChordal} is not true even if we replace the graph $G$ with a suitable one. We provide an example to demonstrate this.

\begin{Ex}\label{Ex:3SecbutNot2Sec}
Let $G$ be the graph illustrated in \Cref{fig:3SecbutNot2Sec}. One can see that $G$ is $3$-secant chordal but not $2$-secant chordal as it contains $C_5$ as an induced subgraph. However, it is easy to check that $G$ is the unique graph yielding the same $\sigma_q\Delta(G)$.

\begin{figure}[h!]
$$
\begin{tikzcd}[column sep = 1 em]
    && \circ \ar[dll,no head] \ar[drr,no head] \\
    \circ \ar[dr,no head] & \circ \ar[d,no head] & & \circ \ar[d,no head] & \circ \ar[dl,no head] \\
    \circ \ar[r,no head] & \circ \ar[rr,no head] & & \circ \ar[r,no head] & \circ 
\end{tikzcd}
$$
    \caption{A graph that is $3$-secant chordal but not $2$-secant chordal}
    \label{fig:3SecbutNot2Sec}
\end{figure}
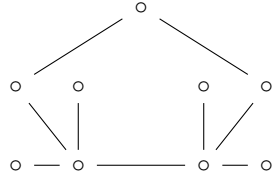

\begin{table}[h!]
    \centering
    \begin{tabular}{rl}
    $S(\sigma_2\Delta(G))$:\quad & \begin{tabular}{c|C{3ex}C{3ex}C{3ex}C{3ex}C{3ex}C{3ex}C{3ex}}
         $j \setminus i$ & 0 & 1 & 2 & 3 & 4 & 5 & 6\\
       \hline
      0 & 1 & . & . & . & . & . & .\\
      1 & . & . & . & . & . & . & .\\
      2 & . & 36 & 111 & 140 & 86 & 24 & 2\\
      3 & . & . & . & . & . & . & .\\
      4 & . & 1 & 4 & 6 & 4 & 1 & .
      \end{tabular} \\
      [-0.5em]\\
        $S(\sigma_3\Delta(G))$:\quad & \begin{tabular}{c|C{3ex}C{3ex}C{3ex}C{3ex}C{3ex}}
         $j \setminus i$ & 0 & 1 & 2 & 3 & 4\\
       \hline
      0 & 1 & . & . & . & .\\
      1 & . & . & . & . & .\\
      2 & . & . & . & . & .\\
      3 & . & 23 & 48 & 34 & 8
      \end{tabular}
    \end{tabular}
    \caption{Betti tables of $S(\sigma_2\Delta(G))$ and $S(\sigma_3\Delta(G))$ for \Cref{Ex:3SecbutNot2Sec}}
\end{table}
\end{Ex}

\Cref{F2toF1} naturally gives rise to the following problem. Refer to an observation that for all integers $q_1,q_2\geq 2$ if $G_1\in\mathcal F_{q_1,1}$ and $G_2\in\mathcal F_{q_2,1}$ are given, then $G_1\sqcup G_2\in\mathcal F_{q_1+q_2+1,1}$ by \Cref{Isol}.

\begin{Prob}
Let $q\geq 3$ be an integer. Find and characterize connected graphs $G$ lying in $\mathcal F_{q,1}$ but not appearing as edge contractions of elementary bipartite graphs on $2q+2$ vertices.
\end{Prob}

From \cite[Theorem 1.1]{MR4441153} it follows that for a nondegenerate projective variety $X\subseteq\mathbb P^r$ if $\sigma_qX$ is $(q+1)$-regular, then $S(\sigma_qX)$ is always Cohen-Macaulay. 
However, a natural counterpart in combinatorics does not hold. 
For instance, there are many chordal graphs $G$ for which $S(\Delta(G))$ is not Cohen-Macaulay. See \cite{MR2231097} for information on such chordal graphs.

\begin{Prob}\label{Prob:CM}
Let $G$ be a $q$-secant chordal graph for an integer $q\geq 1$. Find a necessary and sufficient condition for $\sigma_q\Delta(G)$ to be Cohen-Macaulay in terms of graph theory. What about the shellability?
\end{Prob}

\Cref{cor:prdim} would be a strong hint by producing examples and nonexamples of $q$-secant chordal graphs whose $\sigma_q\Delta(G)$ are Cohen-Macaulay.

For forests $G$ and their $q$-secant complexes $\sigma_q\Delta(G)$ we may establish lower bounds for the difference of projective dimension and codimension via the number of connected components, the array of degrees of vertices, the arrangement of branch vertices, etc.

\begin{Prob}
Let $G$ be a forest and $q\geq 1$ be an integer. Find a nice sharp lower bound for $\pd S(\sigma_q\Delta(G))-\codim\sigma_q\Delta(G)$.
\end{Prob}

\bibliographystyle{amsalpha}
\bibliography{ref}

\end{document}